\documentclass[10pt]{article}
\usepackage{delarray} 
\usepackage{amscd}
\usepackage{amssymb}
\usepackage{amsmath}
\usepackage{amsthm,amsxtra}
\usepackage{latexsym,amsmath,mathrsfs}
\usepackage[bookmarks,colorlinks]{hyperref}
\usepackage[all]{xy}
\usepackage{amsfonts}
\usepackage{color}
\usepackage{fancybox}
\usepackage{colordvi}
\usepackage{multicol}
\usepackage{textcomp}
\usepackage{stmaryrd}
\usepackage[dvips]{graphicx}
\usepackage{enumerate,xspace}
\usepackage{geometry}
\usepackage{tocbibind}

\def\BZ{\mathbb{Z}}
\def\BQ{\mathbb{Q}}

\def\GL{\mathrm{GL}} \def\res{\mathrm{res}}

\def\fin{\mathrm{fin}} \def\sing{\mathrm{sing}}
\def\ord{\mathrm{ord}} \def\Sel{\mathrm{Sel}}
 
\def\ord{\mathrm{ord}} \def\Gal{\mathrm{Gal}}
\def\BC{\mathbb{C}}

\newcommand{\Hom}{\mathrm{Hom}}

    \theoremstyle{plain}

    \newtheorem{thm}{Theorem}[section] \newtheorem{cor}[thm]{Corollary}
    \newtheorem{lem}[thm]{Lemma} 
    \newtheorem{prop}[thm]{Proposition}
    \newtheorem {conj}[thm]{Conjecture}

    \theoremstyle{definition}

    \newtheorem{defn}[thm]{Definition}

    \theoremstyle{remark}
    \newtheorem {rem}[thm]{Remark}

    \numberwithin{equation}{section}

\newcommand{\wvec}[4]{{\scriptsize{\big ( \!\!
\begin{array}{cc} #1 \!\!\! & \!\!\! #2 \\ #3 \!\!\! & \!\!\! #4 \end{array} \!\! \big ) }}}

\begin{document}

\title{Comparison of canonical periods under base change}
\author{Qingshen Lv,\\ \small School of Mathematics and Statistics, University College Dublin
 \\ \small qingshen.lv@ucdconnect.ie
 \\ Bingyong Xie \footnote{The second author is supported by
the National Natural Science Foundation of China (grant 12231001),
and by Science and Technology Commission of Shanghai Municipality
(no. 22DZ2229014).}
 \\ \small School of Mathematical Sciences, \& Key Laboratory of MEA (Ministry of Education) \\ \small \&   Shanghai Key Laboratory of PMMP,  East China Normal University
 \\ \small byxie@math.ecnu.edu.cn}
\date{} \maketitle

\section{Introduction}

\subsection{Special value formula, Gross period and Hida canonical period}

In \cite{Gro87} Gross provided a formula expressing the special
values of $L$-functions via Heegner points. To describe his result,
let $f$ be a newform on $\Gamma_0(N)$  of even weight $k$, and let $K$ be an imaginary
quadratic field of discriminant $D$ such that $D$ is prime to $N$.

Assume $k=2$ and $N$ is a prime inert in $K$. Let $B$ be the definite quaternion algebra that is ramified exactly at $N$. Fix a prime number $p$. Let $H_n$ be the ring class field of $K$ of conductor $p^n$. When $\chi$ is a character of the Galois group $\mathrm{Gal}(H_n/K)$, one forms the Rankin-Selberg $L$-function $L(f,\chi,s)$.
Gross \cite{Gro87} showed that there exists a period $\Omega_{f,K}$
depending on $f$ and $K$ but independent of $n$ and $\chi$, called Gross
period, such that $\frac{L(f,\chi,1)}{\Omega_{f,K}}$ are algebraic
and they satisfy
$$\frac{L(f,\chi,1)}{\Omega_{f,K}}\cdot\sqrt{D}p^{n}=|\sum\limits_{\sigma\in\mathrm{Gal}(H_n/K)}\chi(\sigma)\psi(P^{\sigma})|^{2},$$
where $P$ is the Heegner point on the Shimura set attached to $B$.

Gross formula is generalized by Shouwu Zhang \cite[Theorem 7.1]{Zha04} for $k=2$ and general $N$, and by Haiping Yuan \cite{Yua05} for $k>2$.

In \cite{BD07}  Bertolini and Darmon used Gross' special value
formula to construct anticyclotomic $p$-adic $L$-functions, i.e.
$p$-adic $L$-functions for anticyclotomic $\BZ_p$-extension instead
of the cyclotomic $\BZ_p$-extension. In the same paper they proved
one divisibility for Iwasawa main conjecture for elliptic curves in
the setting of anticyclotomic $\BZ_p$-extension (\cite[Theorem
1]{BD07}), and also gave new evidence for Birch and Swinnerton-Dyer
conjecture (\cite[Corollary 4]{BD07}).

Gross period serves as a bridge between the complex $L$-function and
the $p$-adic $L$-function. But it depends on the choice of an imaginary
quadratic extension $K$.

Hida \cite{Hid81, Hid81.2} introduced a period that is independent
of $K$ and called the {\it canonical period}. Attached to $f$ there
exists a homomorphism $\lambda_f:T\longrightarrow \mathcal{O}_f$
from the Hecke algebra $T$ of level $N$ to a discrete valuation ring
$\mathcal{O}_f$ that is finite over $\BZ_p$. Let $\eta $ be the
congruence number for $\lambda_f$ defined by
$$\eta =\lambda_f(\mathrm{Ann}(\mathrm{Ker}(\lambda_f))).$$ Then the
canonical period is defined by
$$\Omega_{f}^{\mathrm{can}}=\frac{\langle f, f
\rangle_{\mathrm{Pet}}}{\eta }.$$ The reader may consult
\cite[Section 2]{Vat03} about more knowledge on Gross period and the
canonical period.

Vatsal \cite{Vat03} showed that
$\frac{\Omega_{f,K}}{\Omega_{f}^{\mathrm{can}}}$ lies in
$\mathcal{O}_f$. Pollack and Weston \cite{PW11} pointed out that
$\frac{\Omega_{f,K}}{\Omega_{f}^{\mathrm{can}}}\in\mathcal{O}_f^{\times}$
is equivalent to the freeness of spaces of modular forms on $B$ over
the associated Hecke algebra and the vanishing of certain local
Tamagawa components. In \cite{CH18} Chida and Hsieh actually proved
that $\frac{\Omega_{f,K}}{\Omega_{f}^{\mathrm{can}}}$ is a $p$-adic
unit, i.e.
$\frac{\Omega_{f,K}}{\Omega_{f}^{\mathrm{can}}}\in\mathcal{O}_f^{\times}$
under a condition ($\mathrm{CR^{+}}$). Wang \cite{Wang} 
generalized Chida and Hsieh's result to the setting of Hilbert
modular forms under similar hypothesis. 

There are two normalizations of anticyclotomic $p$-adic
$L$-functions depending on which period one uses. However, in the
present paper we assume those hypothesis in \cite{Wang} in order
that there is essentially no difference between these two
normalizations.

\subsection{Comparison of periods under base change}

The goal of the present paper is to compare the canonical period of
a Hilbert modular form and that of its base change to a (totally)
real quadratic extension.

To make it precise, let $F$ be a totally real number field, and  $f$
be a Hilbert newform of parallel even weight over $F$. Let $F'$ be a
totally real quadratic extension of $F$, and $f'$ be the base change
of $f$ to $F'$. One may expect the following holds.

\begin{conj} We have
$$ \Omega_{f'}^\mathrm{can}=(\Omega_f^\mathrm{can})^2 $$ up to a
$p$-adic unit.
\end{conj}

The main result of this paper is the following theorem. Let $\rho_f$ be $p$-adic Galois representation attached to $f$ (see \cite{Wil88,Tay}), and let $\bar{\rho}_f$ denote the residue representation of $\rho_f$.

\begin{thm}\label{thm:main} Let $F$ be a totally real field,
and $F'$ be a real quadratic extension of $F$. Assume that $p$ is
unramified in both $F$ and $F'$. Let $f$ be a Hilbert newform of
parallel even weight $k$ over $F$ with trivial central character and
level $\mathfrak{n}_f$. Let $f'$ be the base change of $f$ to $F'$.
Let $\mathfrak{n}_{f'}$ be the level of $f'$.

Let $p$ be a prime number satisfying $p\geq\mathrm{max}(k+2,7)$. We
assume that $\mathfrak{n}_f$ is prime to $p$ and $f$ is ordinary at
each prime above $p$. When $[F:\mathbb{Q}]$ is odd, we assume that
there exists at least one prime $\mathfrak{q}$ such that
$\mathfrak{q}|| \mathfrak{n}_f$, and $\mathfrak{q}$ is split in
$F'$.

Suppose the following conditions hold.
\begin{enumerate}
\item The restrictions of $\bar{\rho}_f$ to $G_{F'(\xi_p)}$ and $G_{F'(
\sqrt{ p^*})}$ are absolutely irreducible, where
$p^*=(-1)^{\frac{p-1}{2}}p$.
\item In the case of $k=2$, for each place $v$ of $F$ above $p$ we
have  $a_v^2(f)\ {\backslash\hskip -10pt \equiv } 1 \ (\mathrm{mod}
\ p)$; similarly, for each place $v'$ of $F'$ above $p$ we have
$a_{v'}^2(f')\ {\backslash\hskip -10pt \equiv } 1 \ (\mathrm{mod} \
p)$.
\item For each $\mathfrak{l}|\mathfrak{n}$, if
$\bar{\rho}_f|_{I_{F_\mathfrak{l}}}$ is absolutely irreducible, then
$\mathrm{N}(\mathfrak{l})\equiv  {\hskip -10pt /}-1 \ (\mathrm{mod
}\ p)$. Similarly, for each $\mathfrak{l}'|\mathfrak{n}'$, if
$\bar{\rho}_{f'}|_{I_{F'_{\mathfrak{l}'}}}$ is absolutely
irreducible, then $\mathrm{N}(\mathfrak{l}')\equiv  {\hskip -10pt
/}-1 \ (\mathrm{mod }\ p)$.
\item $\rho_f$ is a minimal modular lifting of $\bar{\rho}_f$.
\end{enumerate} Then $$
\Omega_{f'}^\mathrm{can}=(\Omega_f^\mathrm{can})^2 $$ up to a
$p$-adic unit.
\end{thm}


We choose an imaginary quadratic extension $K$ of $F$ that is
inertia at $\mathfrak{q}$ when $[F:\mathbb{Q}]$ is odd. Denote $KF'$
by $K'$. Fix a place $\mathfrak{p}$ of $F$  above $p$. 
Let $K_\infty$ be the $\mathfrak{p}$-anticyclotomic
$\mathbb{Z}_p$-extension of $K$.

Using the canonical periods one attaches to $f$,
$f\otimes\chi_{F'/F}$ and $f'$ the $p$-adic $L$-functions
$$L_{p}(K_{\infty} ,f), \
L_{p}(K_{\infty} ,f\otimes\chi_{F'/F})  \ \text{ and }
L_{p}(K'K_{\infty} ,f').$$ All of them can be considered as elements
of $\mathcal{O}_{\mathfrak{P}}[[\mathrm{Gal}(K_{\infty} /K)]]$. The
ring $\mathcal{O}_{\mathfrak{P}}\supset \mathcal{O}_f$ is some
coefficient ring that is clearly defined in our context. To prove
Theorem \ref{thm:main} we only need to show
\begin{equation}\label{eq:main}
(L_{p}(K'K_{\infty} ,f'))=(L_{p}(K_{\infty} ,f))\cdot
(L_{p}(K_{\infty} ,f\otimes\chi_{F'/F}))  \end{equation} in
$\mathcal{O}_{\mathfrak{P}}[[\mathrm{Gal}(K_{\infty} /K)]].$

Our strategy is to use Iwasawa main conjecture that provides an
equality between a quantity measuring Selmer groups and $p$-adic
$L$-functions. Actually what we need is Iwasawa main conjecture for
Hilbert modular forms in the anticyclotomic setting.

Its proof is divided into two parts, one part proving one
divisibility by Ribet's method, and the other proving the converse
divisibility by Euler systems. The former divisibility was proved by
Skinner and Urban \cite{S-U} for elliptic modular forms, and was
proved by Wan \cite{Wan} for Hilbert modular forms. When
$[F:\mathbb{Q}]$ is odd, Wan needs the condition that Ihara Lemma
for Shimura curves holds. The latter divisibility was proved by
Bertolini and Darmon \cite{BD07} for elliptic curves, and by Chida
and Hsieh \cite{CH15} for elliptic modular forms. It was proved by
Longo \cite{Longo} and Wang \cite{Wang} for Hilbert modular forms
assuming Ihara Lemma for Shimura curves. This condition was removed
by the second author \cite{Xie}.

Combining results in \cite{Wan} and \cite{Xie} we obtain the
following theorem. Remark that we do not need Ihara Lemma for
Shimura curves even when $[F:\mathbb{Q}]$ is odd, though such a
condition is needed in \cite{Wan}.

\begin{thm}\label{thm:Iw-main-int}
Assume that $f$ satisfies conditions $(\mathrm{CR}^+)$,
$(\mathfrak{n}^+\text{-}\mathrm{DT})$, $(\mathrm{PO})$ and $(\mathrm{Fuji}1$-$4)$.
Then we have $$(L_p(K_\infty, f))=\mathrm{char}\ \mathrm{Sel}
(K_\infty, A_{\rho})^\vee .$$
\end{thm} See our context for Conditions $(\mathrm{CR}^+)$,
$(\mathfrak{n}^+\text{-}\mathrm{DT})$, $(\mathrm{PO})$ and
$(\mathrm{Fuji}1$-$4)$.

By Theorem \ref{thm:Iw-main-int} we have $$
(L_{p}(K_{\infty},f))=\mathrm{char}\
\mathrm{Sel}(K_{\infty},A_{\rho})^{\vee} $$ and
    $$  (L_{p}(K_{\infty},f\otimes\chi_{F'/F}))=\mathrm{char}\ \mathrm{Sel}(K_{\infty},A_{\rho\otimes\chi_{F'/F}})^{\vee}    .$$
Using the relation
$$\mathrm{Sel}(K'K_{\infty},A_{\rho})=\mathrm{Sel}(K_{\infty},A_{\rho})\bigoplus\mathrm{Sel}(K_{\infty},A_{\rho\otimes\chi_{F'/F}})$$
we obtain \begin{equation}\label{eq:Lp-eq} (L_{p}(K_{\infty},f)\cdot
L_{p}(K_{\infty},f\otimes\chi_{F'/F}))= \mathrm{char}\
\mathrm{Sel}(K'K_{\infty},A_{\rho})^{\vee}. \end{equation}

Let $\mathfrak{p}'$ be a place of $F'$ above $\mathfrak{p}$. Let
$K'_\infty$ be the $\mathfrak{p}'$-anticyclotomic
$\mathbb{Z}_p$-extension of $K'$.    With $K'_\infty$ instead of
$K_\infty$ we also have $$ (L_{p}(K_{\infty}',f'))=\mathrm{char}\
\mathrm{Sel}(K_{\infty}',A_{\rho})^{\vee},$$ from which we deduce
that \begin{equation}\label{eq:con}(L_{p}(K'K_{\infty},f'))\supseteq
\mathrm{char}\ \mathrm{Sel}(K'K_{\infty},A_{\rho})^{\vee}.
\end{equation} Combining (\ref{eq:Lp-eq}) and (\ref{eq:con}) we
obtain
\begin{equation}\label{eq:contain} (L_{p}(K'K_{\infty},f'))\supseteq (L_{p}(K_{\infty},f)\cdot
L_{p}(K_{\infty},f\otimes\chi_{F'/F})) . \end{equation} As
(\ref{eq:main}) holds in
$\mathcal{O}_{\mathfrak{P}}[[\mathrm{Gal}(K_{\infty}
/K)]][\frac{1}{p}]$, and as the $\mu$-invariants of
$L_{p}(K_{\infty} ,f)$ and $L_{p}(K_{\infty} ,f\otimes\chi_{F'/F})$
are zero, (\ref{eq:contain}) implies (\ref{eq:main}).

The paper is organized as follows. In Section \ref{sec:a} we collect
some basic facts about automorphic forms on definite quaternion
algebras. In Section \ref{sec:b} we recall the construction of
anticyclotomic $p$-adic $L$-functions via Theta elements. In Section
\ref{sec:d}, we combine Wan's result in \cite{Wan} with the result
in \cite{Xie} to prove Iwasawa main conjecture for Hilbert modular
form in the anticyclotomic setting (under certain conditions).
Finally we prove Theorem \ref{thm:main} in Section \ref{sec:e}.

In Sections \ref{sec:a}-\ref{sec:d}, the imaginary quadratic
extension $K$ of $F$ is fixed. In Section \ref{sec:e}, we will
choose a suitable $K$.

The authors thank B.H. Gross for his helpful advice.

\subsection*{Notations}

We fix a totally real number field \(F\) over \(\mathbb{Q}\). Let
\(\Sigma_{F}\) be the set of all real embeddings of \(F\) and let
\(\Sigma_{p}\) be the set of all places of \(F\) above \(p\).

For a number field \(L\), and each place \(v\) of \(L\) above a
prime number \(l\), let \(|\cdot|_{v}\) or \(|\cdot|_{L_{v}}\) be
the absolute value on \(L_{v}\) defined by
\(|x|_{v}=|N_{L_{v}/\mathbb{Q}_{l}}(x)|_{l}\), and  let
\(\varpi_{v}\) be a uniformizer of \(\mathcal{O}_{L_{v}}\).

Let $E$ be a finite extension of $\BQ$, $\mathcal{O}$ be the ring of
integers in $E$ and $\mathfrak{P}$ be a prime  of $\mathcal{O}$
above $p$ such that $\mathcal{O}_\mathfrak{P}$ contains
$\mathcal{O}_f$. Let $\omega$ be a uniformizer of
$\mathcal{O}_\mathfrak{P}$. For each positive integer $n$ we put
$\mathcal{O}_{n}=\mathcal{O}_\mathfrak{P}/\omega^n$. We consider
$E_\mathfrak{P}$, $\mathcal{O}_\mathfrak{P}$ and $\mathcal{O}_n$ as
coefficient rings, and let $G_F=\mathrm{Gal}(\overline{F}/F)$ act trivially on them.

Let $\epsilon$ be the $p$-adic cyclotomic character of $G_F$.

\section{Automorphic representations on definite quaternion
algebras}\label{sec:a}

Let $f$ be a Hilbert cusp newform of even parallel weight $k$ and
trivial central character. Let $\{\mathbf{a}_v\}_v$ be the system of
Hecke eigenvalues attached to $f$.

Let \(\pi\) be the unitary irreducible cuspidal automorphic
representation on \(\text{GL}_2(\mathbf{A}_{F})\) corresponding to
$f$. \noindent We write the conductor of \(\pi\) in the form
\[\mathfrak{n}=\prod\mathfrak{n}^{+}\mathfrak{n}^{-},\]
where \(\mathfrak{n}^{+}\) (resp. \(\mathfrak{n}^{-}\)) is only
divisible by primes that are split (resp. inert or ramified) in
\(K\). We assume the following conditions hold.
\begin{quote}
$\bullet$ \((\mathfrak{n},p)=1\). \\
$\bullet$ \(\mathfrak{n}^{-}\) is square-free, and the number of
prime factors of \(\mathfrak{n}^{-}\) has the same parity as
\([F:\mathbb{Q}]\).
\\ $\bullet$ If \(v \, | \,
\mathfrak{n}^{-}\),  \(\pi_{v}\) is a special representation
\(\sigma(\mu_v,\mu_v|\cdot|_{v}^{-1})\) with unramified character
\(\mu_v\).
\end{quote}

Let \(B\) be the definite quaternion algebra over \(F\) with
discriminant \(\mathfrak{n}^{-}\). Let $G$ be the algebraic group
$\mathrm{Res}_{F/\mathbb{Q}}B^\times$. For each $\sigma\in
\Sigma_F$, the isomorphism $$(B\otimes
_{\sigma(F)}\mathbb{R})\otimes_{\mathbb{R}}\mathbb{C}\cong
M_2(\mathbb{C})$$ induces an embedding $$\iota_\sigma:
G(\mathbb{R}_\sigma)\hookrightarrow\GL_2(\mathbb{C}).$$ Here, the
index $\sigma$ in $\mathbb{R}_\sigma$ indicates that $F$ is embedded
into $\mathbb{R}$ via $\sigma$.

Let \(\pi'=\otimes\pi_{v}'\) be the unitary irreducible cuspidal
automorphic representation on \(G(\mathbf{A}_{F})\) with trivial
central character attached to \(\pi\) via Jacquet-Langlands
correspondence. Then the archimedean constituent of $\pi'$ is
isomorphic to \( (\rho_{k,\infty},V_{k,\infty}(\mathbb{C}))\). We
describe \((\rho_{k,\infty},V_{k,\infty}(\mathbb{C}))\) as follows.

For a commutative ring \(A\), we denote by
\(V_{k}(A)=\text{Sym}^{k-2}(A^2)\) the set of homogeneous
polynomials of degree \(k-2\) with coefficients in \(A\). Write
\[V_{k}(A)=\bigoplus\limits_{-\frac{k}{2}<m<\frac{k}{2}}A\cdot \textbf{v}_{m} \ \ \ (\textbf{v}_m:=X^{\frac{k-2}{2}-m}Y^{\frac{k-2}{2}+m}).\]
Let \(\rho_{k}:\text{GL}_2(A)\rightarrow \text{Aut} V_{k}(A)\) be
the unitary representation defined by
\[\rho_{k}(g)P(X,Y)=\text{det}(g)^{-\frac{k-2}{2}}\cdot P((X,Y)g) \ \ \ (P(X,Y)\in V_{k}(A)).  \]
If \(A\) is a \(\mathbb{Z}_{(p)}\)-algebra with \(p>k-2\), we define
a perfect pairing \[\langle \cdot , \cdot \rangle_{k}:V_{k}(A)\times
V_{k}(A)\longrightarrow A\] by
\[\langle\sum\limits_{-\frac{k}{2}<i<\frac{k}{2}}a_{i}\textbf{v}_{i},
\sum\limits_{-\frac{k}{2}<j<\frac{k}{2}}b_{j}\textbf{v}_{j} \
\rangle_{k}
=\sum\limits_{-\frac{k}{2}<m<\frac{k}{2}}a_{m}b_{-m}\cdot
(-1)^{\frac{k-2}{2}+m}\frac{\Gamma(\frac{k}{2}+m)\Gamma(\frac{k}{2}-m)}{\Gamma(k-1)}.\]
This pairing is \(\text{GL}_2(A)\)-invariant, which means
\[\langle \, \rho_{k}(g)P,\rho_{k}(g)P'\rangle_{k}=\langle P,P'\rangle_{k} \ \ \ (\forall \ P,P'\in V_{k}(A)).\]

For each $\sigma\in\Sigma_F$ put $k_\sigma=k$. Put
\(V_{k,\infty}(\mathbb{C})=\otimes_{\sigma\in\Sigma_{F}}V_{k_{\sigma}}(\mathbb{C})\)
and let \(\rho_{ k,\infty }\) be the representation of
\(G(\mathbb{R})\) defined by
\[\rho_{k,\infty}:G(\mathbb{R})\xrightarrow[]{\prod_{\sigma}\iota_{\sigma}}\prod\limits_{\sigma}\text{GL}_2(\mathbb{C})
\xrightarrow[]{\prod_{\sigma}\rho_{k_{\sigma}}}\text{Aut}V_{k,\infty}(\mathbb{C}).\]
For \(\mathbf{m}=\sum\limits_{\sigma\in\Sigma_{F}}m_{\sigma}\sigma\)
with \(-\frac{k }{2}<m_{\sigma}<\frac{k }{2}\), let
\(\textbf{v}_\mathbf{m}\) denote the element
\(\otimes_{\sigma}\textbf{v}_{m_{\sigma}}\) in
\(V_{k,\infty}(\mathbb{C})\).  We define \(\langle \cdot , \cdot
\rangle_{k,\infty}\) to be the pairing induced from \(\{\langle
\cdot , \cdot \rangle_{k_{\sigma}}\}_{\sigma\in\Sigma_{F}}\).

\begin{defn}
If \(A\) is a \(K\)-algebra and \(U\subset
G(\mathbf{A}_{F}^\infty)\) is an open compact subgroup. A $A$-valued
modular form on \(G(\mathbf{A}_{F}^\infty)\) of trivial central
character, parallel weight \(k\) and level \(U\) is a function \[h:
G(\mathbf{A}_{F}^\infty)\longrightarrow V_{k,\infty}(A)\] that
satisfies
\[h(z\gamma bu)=\rho_{k,\infty}(\gamma)h(b)\]
for all \(\gamma\in B^\times, \ u\in U, \ b\in\widehat{B}^{\times}\)
and \(z\in\widehat{F}^{\times}\). We denote by \(M_{k}^{B}(U,A)\)
the space of such forms. The right translation makes
\(M_{k}^{B}(A):=\varinjlim\limits_{U}M_{k}^{B}(U,A)\) an admissible
\(G(\mathbf{A}_{F}^\infty)\)-representation.\end{defn} For
\(\textbf{v}\in V_{k,\infty}(\mathbb{C})\) and \( h\in
M_{k}^{B}(\mathbb{C})\), we can attach to \(\textbf{v}\otimes h\) an
automorphic form \(\Psi(\textbf{v}\otimes h)\) on
\(G(\mathbf{A}_{F})\) by
\[\Psi(\textbf{v}\otimes h)(g):=\langle \rho_{k,\infty}(g_{\infty})\textbf{v} ,h(g^{\infty}) \rangle_{k,\infty}.\]
Let \(\mathcal{A}_{k}^{B}(U,\mathbb{C})\) be the space of
automorphic forms on \(G(\mathbf{A}_{F})\)
consisting of functions \(\Psi(\textbf{v}\otimes h)\) for \(h\in M_{k}^{B}(U,\mathbb{C})\) and \(\textbf{v} \in V_{k,\infty}(\mathbb{C})\).\\

Realize $\pi'^{\infty}$ in $M^B_{k}(\BC)$ so that
$\varphi^\infty:=\bigotimes\limits_{v: \text{finite} }\varphi_v$ is
an element of $M^B_{k}(\BC)$. Let $J$ be a nonempty subset of
$\Sigma_p$. For each $\mathfrak{p}\in J$, let $\alpha_\mathfrak{p}$
be the unit root of
$X^2-\mathbf{a}_\mathfrak{p}X+N(\mathfrak{p})^{k-1}$. We put
$$ \varphi^{\infty , \dagger_J}= \varphi^\infty-
\Big(\prod_{\mathfrak{p}\in J}\frac{1}{\alpha_\mathfrak{p}}
\pi'(\wvec{1}{0}{0}{\varpi_\mathfrak{p}})\Big)\varphi^\infty $$ and
$$\varphi^{\dagger_J}_{\pi'}= \Psi( \mathbf{v}_{\mathbf{0}} \otimes
\varphi^{\infty,\dagger_J}) .$$

\section{Anticyclotomic \(p\)-adic \(L\)-functions} \label{sec:b}

\subsection{Optimal embeddings and Gross points}
Write \(B=K\bigoplus KI\) with \(I^{2}=\beta\in F^{\times}\). We may
take \(I\) satisfies the following conditions.
\begin{quote}
$\bullet$ \(\sigma(\beta)<0\) for all \(\sigma\in\Sigma_{F}\).\\
$\bullet$ \(\beta\in(\mathcal{O}_{F_{v}})^{2}\) for all \(v \, | \,
p \mathfrak{n}^+ \), and \(\beta\in\mathcal{O}_{F_{v}}^{\times}\)
for all \(v \, | \, N_{K/F}D_{K}\).
\end{quote}

Fix a square root \(\sqrt{\beta}\in\overline{\mathbb{Q}}^{\times}\)
of \(\beta\). We fix an isomorphism \(i=\prod
i_{v}:\widehat{B}^{(\mathfrak{n}^{-})}\simeq
 M_{2}(\mathbf{A}_{F,f}^{(\mathfrak{n}^{-})})\) as follows.
Fix an element \(\theta\in K\) such that \(\{1,\theta\}\) is a basis
of \(\mathcal{O}_{K_{v}}\) over \(\mathcal{O}_{F_{v}}\) for all \(v
\, | \, p \mathfrak{n}^+ \). For each \(v \, | \, p\mathfrak{n}^+\),
we define the isomorphism \(i_{v}:B_{v}:=B\otimes_{F} F_{v}\simeq
M_{2}(F_{v})\) by
\[i_{v}(\theta)=\begin{pmatrix}
    T(\theta) & -N(\theta) \\
    1         & 0
\end{pmatrix}, \ i_{v}(I)=\sqrt{\beta}\begin{pmatrix}
    -1 & T(\theta)\\
    0  & 1
\end{pmatrix}.\] Here, we write \(T\) and \(N\) for the reduced trace and norm of \(B\).
For each finite place \(v\nmid p\mathfrak{n} \), we fix an
isomorphism \(i_{v}:B_{v}\simeq M_{2}(F_{v})\) such that
\[i_{v}(\mathcal{O}_{K}\otimes\mathcal{O}_{F_{v}})\subset M_{2}(\mathcal{O}_{F_{v}}).\]
For \(\sigma\in\Sigma_{F}\), we define \(i_{\sigma}:B\hookrightarrow
M_{2}(\mathbb{C})\) by
\[a+bI\mapsto i_{\sigma}(a+bI):=\begin{pmatrix}
    \sigma(a) & \sigma(b\beta)\\
    \sigma(\Bar{b}) & \sigma(\Bar{a})
\end{pmatrix} \ (a,b\in K).\]

We fix a decomposition
\(\mathfrak{n}^{+}=\mathfrak{N}^{+}\overline{\mathfrak{N}^{+}}\)
once and for all, we define \(\varsigma_{v}\in G(F_{v})\) for
\(v\notin J\) as follows.
\begin{quote} $\bullet$ \(\varsigma_{v}=1\) If \(v\notin J\) and \(v\nmid \mathfrak{n}^{+}\).\\
$\bullet$ \(\varsigma_{v}=\delta^{-1}\begin{pmatrix}
    \theta & \Bar{\theta}\\
    1      & 1
\end{pmatrix}\) if \(v\notin J\) and \(v=\varpi\Bar{\varpi}\) with \(\varpi \, | \,
\mathfrak{N}^{+}\).\end{quote} For each non-negative integer \(n\)
and for each \(\mathfrak{p}\in J\), we define
\(\varsigma_{\mathfrak{p}}^{(n_{\mathfrak{p}})}\in
G(F_{\mathfrak{p}})\) as follows. \begin{quote} $\bullet$
\(\varsigma_{\mathfrak{p}}^{(n_{\mathfrak{p}})}=\begin{pmatrix}
    \theta & -1\\
    1      & 0
\end{pmatrix}\begin{pmatrix}
    \varpi_{\mathfrak{p}}^{n_{\mathfrak{p}}} & 0\\
    0                       & 1
\end{pmatrix}\in \text{GL}_{2}(F_{\mathfrak{p}})\) for \(\mathfrak{p}\in J\) split in \(K\).\\
$\bullet$
\(\varsigma_{\mathfrak{p}}^{(n_{\mathfrak{p}})}=\begin{pmatrix}
    0 & 1\\
    -1 & 0
\end{pmatrix}\begin{pmatrix}
    \varpi_{\mathfrak{p}}^{n_{\mathfrak{p}}} & 0\\
    0                       & 1
\end{pmatrix}\) for \(\mathfrak{p}\in J\) inert or ramified in
\(K\). \end{quote} Define
\(x_{\vec{n}}:\mathbf{A}_{K}^{\times}\longrightarrow
G(\mathbf{A}_{F})\) by
\[x_{\vec{n}}(a):=a\cdot \varsigma_{J}^{(\vec{n})} \ \ \  (\varsigma_{J}^{(\vec{n})}=\prod\limits_{\mathfrak{p}\in J}\varsigma_{\mathfrak{p}}^{(n_{\mathfrak{p}})}\prod\limits_{v\notin J}\varsigma_{v}).\]
One calls \(\{x_{\vec{n}}(a)\}_{a\in\mathbf{A}_{K}^{\times}}\)
Gross points of conductor \(\prod\limits_{\mathfrak{p}\in
J}\mathfrak{p}^{n_{\mathfrak{p}}}\).

\subsection{Theta elements and $p$-adic $L$-functions}

Let $J$ be a nonempty subset of $\Sigma_{p}$. For each $J$-tuple of
nonnegative integers $\vec{n}=(n_{\mathfrak{p}})_{\mathfrak{p}\in
J}$, let
$$\mathcal{O}_{\vec{n} }:=\mathcal{O}_F +
 \prod_{\mathfrak{p}\in J}\mathfrak{p}^{n_\mathfrak{p}}
\mathcal{O}_K$$ be the order of $K$ of conductor
$\prod_{\mathfrak{p}\in J}\mathfrak{p}^{n_\mathfrak{p}} $. Let
$\widetilde{K}^-_{J,\vec{n}}$ be the ring class field of $K$ of
conductor $\prod_{\mathfrak{p}\in J}\mathfrak{p}^{n_\mathfrak{p}}$
and let $\mathcal{G}_{\vec{n}}^{-}$ be its Galois group.

Put
$\widetilde{K}^-_{J,\infty}=\bigcup\limits_{\vec{n}}\widetilde{K}^-_{J,\vec{n}},
$ where $\vec{n}$ runs over all $J$-tuples. Let
$\mathcal{G}^-_{J}=\mathrm{Gal}(\widetilde{K}^-_{J,\infty}/K)$ be
its Galois group (the complex conjugation acting by $-1$ on the
Galois group). Let $K^-_{J,\infty}$ be the subfield of
$\widetilde{K}^-_{J,\infty}$ such that
$\Gamma^-_{J}:=\mathrm{Gal}(K^-_{J,\infty}/K)$ is the maximal
$\BZ_p$-free quotient of $\mathcal{G}^-_{J}$.

We are particularly interested in two cases, the case of
$J=\Sigma_{p}$ and the case of $J=\{\mathfrak{p}\}$ a single set.
When $J=\Sigma_{p}$, we write $\mathcal{K}^-_{\infty}$ for
$K^-_{\Sigma_{p},\infty}$. Then $\mathcal{K}^-_\infty$ is the
maximal abelian anticyclotomic $\BZ_p$-extension of $K$ unramified
outside $p$ with Galois group denoted as $\Gamma^-_K$, which is
isomorphic to $\BZ_p^{ [F:\BQ]}$. When $J=\{\mathfrak{p}\}$,
$\Gamma^-_J$ is isomorphic to $\BZ_p^{ [F_\mathfrak{p}:\BQ_p]}$.  We
write $\widetilde{K}_{n }=\widetilde{K}^-_{\{\mathfrak{p}\}, n }$,
$\widetilde{K}_{\infty}=\widetilde{K}^-_{\{\mathfrak{p}\}, \infty
}$, $ {K}_n= {K}^-_{\{\mathfrak{p}\}, n}$ and
$K_\infty=K^-_{\{\mathfrak{p}\},\infty}$.

Let $F_\infty$ be the cyclotomic $\BZ_p$-extension of $F$ and put
$\mathcal{K}^+_\infty=F_\infty K$. The Galois group
$\Gamma^+_K:=\mathrm{Gal}(\mathcal{K}^+_\infty / K)$ is isomorphic
to $\BZ_p$. Put
$\mathcal{K}_\infty=\mathcal{K}^+_\infty\mathcal{K}^-_\infty$ and
$\Gamma_K=\mathrm{Gal}(\mathcal{K}_\infty/K)$. Then $\Gamma_K \cong
\BZ_p^{[F:\BQ] +1}$.

Let $\phi\in\pi'$ be a normalized new eigenform. See \cite[Section
4]{Xie21} for the definition of $\phi$. The function
\[a\mapsto \phi'^{\dagger_J}(x_{\vec{n}}(a)) :=\Psi(\mathbf{v}_{\mathbf{0}}\otimes
\phi^{\dagger_J}(a\varsigma_{J}^{(\vec{n})})) \] on
\(\mathbf{A}_{K}^{\times}\) is
\(\mathbf{A}_{F}^{\times}K^{\times}K_{\infty}^{\times}\widehat{\mathcal{O}}_{\vec{n}
}^{\times}\)-invariant. For each \(J\)-tuple
\(\vec{n}=(n_{\mathfrak{p}})_{\mathfrak{p}\in J}\), let \([ \ \
]_{\vec{n} }\) be the composition  map
\begin{displaymath}
    \xymatrix{
       \mathbf{A}_{K}^{\infty,\times}/ K^{\times} \ar[r]^{\text{rec}_{K}} & \text{Gal}(\Bar{K}|K)^{\text{ab}} \ar[r] & \mathcal{G}_{\vec{n} }.
    }
\end{displaymath}
The map  \([ \ \ ]_{\vec{n} }\) factors through  \(X_{\vec{n}
}:=\mathbf{A}_{F}^{\infty,\times}\backslash\mathbf{A}_{K}^{\infty,\times}/K^{\times}\widehat{\mathcal{O}}_{\vec{n}
}^{\times},\).

We define the theta element \(\Theta_{\vec{n} }\) by
\[\Theta_{\vec{n} }=\frac{1}{\prod\limits_{\mathfrak{p}\in J}\alpha_{\mathfrak{p}}^{n_{\mathfrak{p}}}}
\cdot\sum\limits_{a\in X_{\vec{n}
}}\phi'^{\dagger_J}(x_{\vec{n}}(a)) [a]_{\vec{n} }.\] If
\(n_{\mathfrak{p}}'\geq n_{\mathfrak{p}}\) for all \(\mathfrak{p}\in
J\), we write \(\vec{n}'\geq\vec{n}\). When \(\vec{n}'\geq\vec{n}\),
we have a natural quotient map
\[\mathrm{pr}_{\vec{n}',\vec{n}}:\mathcal{G}_{\vec{n}' }\longrightarrow\mathcal{G}_{\vec{n} }.\]

\begin{prop}$\cite{Xie21}$ \label{theta-def} If \(\vec{n}'\geq\vec{n}\), then
\[\mathrm{pr}_{\vec{n}',\vec{n}}(\Theta_{\vec{n}' })=\Theta_{\vec{n} }.\]
\end{prop}

Put \(\widehat{\Theta}_{\vec{n} }=\jmath(\Theta_{\vec{n} }).\) We
define
\[\widehat{\Theta}_{J }:=\varprojlim\limits_{\vec{n}}\widehat{\Theta}_{\vec{n} }\in
\mathcal{O}_{\mathfrak{P}}[[\mathcal{G}_{J }]].\] By abuse of
notation the image of \(\widehat{\Theta}_{J }\) by the projection \(
\mathcal{O}_{\mathfrak{P}} [[\mathcal{G}_{J}]]\longrightarrow
\mathcal{O}_{\mathfrak{P}} [[\Gamma_{J}^{-}]]\) is also denoted as
\(\widehat{\Theta}_{J}\). We define the \(p\)-adic $L$-function of
$\pi $ for the extension \(\mathcal{K}_{\infty}^{-}\) by
\[ L_{p}(\mathcal{K}_{\infty}^{-},f)=\widehat{\Theta}_{J}^{2}. \]
For a closed subextension \(\mathcal{L}\) of \(K\) contained in
\(\mathcal{K}_{\infty}^{-}\), we define
\[L_{p}(\mathcal{L},f)=\mathrm{pr}_{\mathcal{L}}(L_{p}(\mathcal{K}_{\infty}^{-},f))\] where
\(\mathrm{pr}_{\mathcal{L}}\) is the projection map
\(\mathcal{O}_{\mathfrak{P}}[[\mathrm{Gal}(\mathcal{K}_{\infty}^{-}/K)]]\rightarrow\mathcal{O}_{\mathfrak{P}}[[\mathrm{Gal}(\mathcal{L}/K)\)]].
In \cite{Xie21} the second author gave an interpolation formula for
these $p$-adic $L$-functions. When $J=\{\mathfrak{p}\}$ is single,
$L_{p}(K_\infty,f)$ coincides with the $p$-adic $L$-function used in
\cite{Xie}.

\section{Selmer groups and Iwasawa main conjecture}\label{sec:d}

Let $f$ be a Hilbert newform of parallel even weight $k$. We need
the conditions $(\mathrm{CR}^+)$,
$(\mathfrak{n}^+\text{-}\mathrm{DT})$ and $(\mathrm{PO})$ used in
\cite{Xie}.

  $(\mathrm{CR}^+)$  1. $p>k+1$ and
$(\#(\mathcal{O}_{F}/\mathfrak{p})^\times)^{k-1}>5$.

2. The restriction of $\bar{\rho}_f$ to $G_{F(\sqrt{p^*})}$ is
absolutely irreducible, where $p^*=(-1)^{\frac{p-1}{2}}p$.

3. If
$\mathfrak{l}|\mathfrak{n}^-$, then $\bar{\rho}_f$ is ramified at $\mathfrak{l}$. 

4. If $\mathfrak{n}_{\bar{\rho}}$ denotes the Artin conductor of
$\bar{\rho}_{f}$, then $\mathfrak{n}/\mathfrak{n}_{\bar{\rho}}$ is
coprime to $\mathfrak{n}_{\bar{\rho}}$. \vskip 5pt

  $(\mathrm{PO})$  If $k=2$, then $a_v^2(f)\ {\backslash\hskip -10pt
\equiv } 1 \ (\mathrm{mod} \ p)$ for all $v|p$ . \vskip 5pt

  $(\mathfrak{n}^+\text{-}\mathrm{DT})$  If
$\mathfrak{l}||\mathfrak{n}^+$ and $\mathrm{N}(\mathfrak{l})\equiv 1
\ (\mathrm{mod} \ p)$, then $\bar{\rho}_f$ is ramified at
$\mathfrak{l}$. \vskip 5pt

For the purpose of applying Wan's result \cite[Theorem 3]{Wan}, we
also need the following conditions: 

  (Fuji1) For each $\mathfrak{l}|\mathfrak{n}$, if
$\bar{\rho}_f|_{I_{F_\mathfrak{l}}}$ is absolutely irreducible, then
$\mathrm{N}(\mathfrak{l})\equiv  {\hskip -10pt /}-1 \ (\mathrm{mod
}\ p)$.

  (Fuji2) There is a minimal modular lifting of
$\bar{\rho}_f$.

 (Fuji3) When $p=5$, the following case is excluded: the projective
image $\bar{G}$ of $\bar{\rho}_f$ is isomorphic to
$\mathrm{PGL}_2(\mathbb{F}_p)$ and the mod $p$ cyclotomic character
$\bar{\chi}$ factors through $\bar{G}^{\mathrm{ab}}\simeq \BZ/2\BZ$.

(Fuji4) $\bar{\rho}_f|_{F(\xi_p)}$ is absolutely irreducible.

Condition $(\mathrm{CR}^+$$3)$ is slightly stronger than that in
\cite{Xie}. Such a stronger version is needed in \cite[Theorem
3]{Wan}. Conditions in \cite[Theorem 3]{Wan} that are already
overlap with $(\mathrm{CR}^+)$,
$(\mathfrak{n}^+\text{-}\mathrm{DT})$ and $(\mathrm{PO})$ are not
included in (Fuji1-4).

\subsection{Selmer groups} \label{ss:selmer}

Let $L$ be a number field. For a finite place $v$ of $L$ and a
discrete $G_L$-module $M$, define the finite part of $H^1(L_v,M)$ by
$$H^1_\fin(L_v,M):= M^{I_{L_v}}/(\mathrm{Frob}_v-1)\cong \mathrm{ker}(H^1(L_v,M)\rightarrow
H^1(I_{L_v},M)),$$ and define the singular quotient of $H^1(L_v,M)$
by $$H^1_\sing(L_v,M):= H^1(L_v,M))/H^1_\fin(L_v,M).$$ Let $\rho=\rho^*_f=\rho_f\otimes \epsilon^{(2-k)/2}$ be the self-dual twist of $\rho_f$, $V_\rho$
the underlying representation space of $\rho$. Let $T_\rho$ be a
$G_F$-invariant lattice in $V_\rho$, and put $A_\rho=V_\rho/T_\rho$.
For each positive integer $n$, put $T_{n}:=T_{\rho}/\varpi^nT_\rho$
and $A_{n}:=\varpi^{-n}T_\rho/T_\rho$. Then
$T_\rho=\lim\limits_{\overleftarrow{\;\:n\:\:}}T_{n}$ and
$A_\rho=\lim\limits_{\overrightarrow{\;\:n\:\:}}A_{n}$.

When $\mathfrak{l}$ is above an $n$-admissible prime, or when
$\mathfrak{l}|p\mathfrak{n}^-$,   $A_{n}$  sits in a $G_{F_\mathfrak{l}}$-equivariant
short exact sequence of free $\mathcal{O}_n$-modules $$ \xymatrix{ 0
\ar[r]  & F^+_\mathfrak{l}A_{ n} \ar[r] & A_{ n} \ar[r] &
F^-_\mathfrak{l}A_{ n} \ar[r] & 0, }
$$
where $G_{F_\mathfrak{l}}$ acts on $F^+_\mathfrak{l}A_{ n}$ (resp. $F^-_\mathfrak{l}A_{ n}$) via $\pm\epsilon $ (resp. $\pm 1$).
One defines the ordinary part of $H^1(L_{\mathfrak{l}}, A_{ n})$ to be
the image of $$ H^1(G_{L_\mathfrak{l}}, F^+_\mathfrak{l}A_{
n})\rightarrow H^1(G_{L_\mathfrak{l}}, A_{ n}).
$$
One defines the ordinary part of $H^1(L_{\mathfrak{l}}, T_{ n})$
similarly \cite{Longo, Wang, Xie}.

When $\mathfrak{l}$ is a finite place of $F$, $M$ is a $G_F$-module,
and $L$ is a finite extension of $F$, for $?=\emptyset,\
\mathrm{fin}, \mathrm{sing}$ or $ \mathrm{ord}$ we put
$$ H^1_?(L_\mathfrak{l}, M) =\bigoplus_{v|\mathfrak{l}} H^1_?(L_v, M),
$$ where $v$ runs over places of $L$ above $\mathfrak{l}$.

Let $\res_v$ denote the restriction map $H^1(L, M)\rightarrow
H^1(L_v, M)$ at $v$.

\begin{defn} $($\cite[Definition 3.1]{Xie}$)$ Let $S$ be a
finite (maybe empty) set of finite places of $F$ that are coprime to
$p\mathfrak{n}$.
 For $M=A_{n}$ or $T_{n}$ we define the Selmer group
$\Sel^S(G_L, M)$ to be the group of elements $s\in H^1(G_L,M)$ such
that

$\bullet$ $\res_v(s)$ is arbitrary if $v$ is above $S$.

$\bullet$ $\res_v(s)\in H^1_\fin(L_v,M)$ if $v\nmid p\mathfrak{n}^-$
and $v$ is not above $S$;

$\bullet$ $\res_v(s)\in H^1_\ord(L_v, M)$ if $v| p\mathfrak{n}^-$
and $v$ is not above $S$.

For $M=A_\rho$ or $T_\rho$, we define $\Sel^S(G_L, M)$ in the same
way.
\end{defn}

By Lemma \ref{lem:p-ord-part} below, if we take $S$ to contain all
primes dividing $\mathfrak{n}$,  then $\Sel^S(G_L, M)$ defined above
coincides with the Selmer group defined in \cite{Wan}.

\begin{lem}\label{lem:p-ord-part} Suppose that $L$ is a finite extension of $F$ contained in $\mathcal{K}_\infty$.
Then the ordinary part of $H^1(L_p, A_n)$ coincides with that
defined in \cite{Wan}.
\end{lem}
\begin{proof}
The ordinary part of $H^1(L_p, A_n)$ defined above coincides with
the kernel of the map
$$H^1(L_p, A_n)\rightarrow \bigoplus_{\mathfrak{q}|p} H^1(G_{L_{\mathfrak{q}}},
F^-_\mathfrak{q}A_n).$$ In \cite{Wan} the ordinary part of $H^1(L_p,
A_n)$ is defined to be the kernel of the map
$$H^1(L_p, A_n)\rightarrow \bigoplus_{\mathfrak{q}|p} H^1(I_{L_{\mathfrak{q}}},
F^-_\mathfrak{q}A_n).$$ The modules $H^1(I_{L_{\mathfrak{q}}},
F^-_\mathfrak{q}A_n)$ and  $H^1(G_{L_{\mathfrak{q}}},
F^-_\mathfrak{q}A_n)$ sit in the following exact sequence
$$\xymatrix{ H^1(G_{L_{\mathfrak{q}}}/I_{L_{\mathfrak{q}}}, (F^-_\mathfrak{q}A_n)^{I_{L_{\mathfrak{q}}}})
\ar[r] & H^1(G_{L_{\mathfrak{q}}} , F^-_\mathfrak{q}A_n) \ar[r] &
H^1(I_{L_{\mathfrak{q}}}, F^-_\mathfrak{q}A_n). }$$ Condition
($\mathrm{PO}$) ensures
$H^1(G_{L_{\mathfrak{q}}}/I_{L_{\mathfrak{q}}},
(F^-_\mathfrak{q}A_n)^{I_{L_{\mathfrak{q}}}})=0$. Thus we have an
inclusion $$ H^1(G_{L_{\mathfrak{q}}} , F^-_\mathfrak{q}A_n)
\hookrightarrow H^1(I_{L_{\mathfrak{q}}}, F^-_\mathfrak{q}A_n). $$
Our lemma follows.
\end{proof}

\begin{lem}\label{lem:control} $($\cite[Proposition 2.10(1)]{Wang} \text{and} \cite[Lemma 3.5 (b)]{Xie}$)$
 Assume  $(\mathrm{CR}^+)$ holds. Let $L/K$ be a finite
extension contained in $\mathcal{K}_\infty$. \begin{enumerate} \item
The restriction maps
$$ \mathrm{Sel} ^S(K, A_{ n})\rightarrow
\mathrm{Sel} ^S(L, A_{ n})^{\mathrm{Gal}(L/K)}
$$ is an isomorphism.
\item
If $S$ contains all primes $\mathfrak{q}|\mathfrak{n}^+$ with
$\bar{\rho}_{f,\mathfrak{q}}$ unramified, then
$$ \mathrm{Sel} ^S(L,A_n)= \mathrm{Sel} ^S(L,A)[\varpi^n] .  $$
\end{enumerate}
\end{lem}
\begin{proof} Proposition 2.10(1) in \cite{Wang} and Lemma 3.5 (b) in \cite{Xie}
are stated under the condition that $L/K$ is a finite extension
contained in $K_\infty$. But the proofs work in our situation.
\end{proof}


Let $\mathcal{L}$ be $\mathcal{K}_\infty$, $\mathcal{K}^-_\infty$,
$K^-_J$ or $K_\infty$. Write $\Gamma=\mathrm{Gal}(\mathcal{L}/K)$.
We put
\begin{eqnarray*} && H^1(\mathcal{L}, A_{n}) = \lim_{\overrightarrow{\:\:L\:\:}} H^1
(L, A_{n}), \hskip 10pt
\widehat{H}^1(\mathcal{L}, T_{n}) = \lim_{\overleftarrow{\:\:L\:\:}} H^1 (L, T_{n}) \\
&& H^1(\mathcal{L}_{\mathfrak{l}}, A_{n}) =
\lim_{\overrightarrow{\:\:L\:\:}} H^1 (L_{\mathfrak{l}}, A_{n})
\hskip 5pt \text{ and } \hskip
5pt\widehat{H}^1(\mathcal{L}_{\mathfrak{l}}, T_{n}) =
\lim_{\overleftarrow{\;\; L\;\;}} H^1 (L_{\mathfrak{l}}, T_{n}),
\end{eqnarray*} where $L$ runs through all finite extensions of $K$ contained in $\mathcal{L}$. The finite parts and the singular quotients
$H^1_?(\mathcal{L}_{\mathfrak{l}}, A_{n})$ and
$\widehat{H}^1_{?}(\mathcal{L}_{\mathfrak{l}}, T_{n})$ for $?\in
\{\fin, \sing\}$ are defined similarly. We define
\begin{eqnarray*} && \mathrm{Sel} ^S(\mathcal{L}, A_{ n}) =
\lim_{\overrightarrow{\;\;L\;\;}} \mathrm{Sel} ^S (L, A_{ n}) \hskip
10pt \text{and} \hskip 10pt \widehat{\mathrm{Sel}} ^S(\mathcal{L},
T_{ n}) = \lim_{\overleftarrow{\;\; L \;\;}}\mathrm{Sel} ^S (L, T_{
n}) .
\end{eqnarray*} We define $\mathrm{Sel}^S(\mathcal{L},
A_{\rho})$ similarly. If $S$ is empty, we drop $S$ from the above
notations. For $M=\mathrm{Sel}^S(\mathcal{L}, A_{\rho})$ or
$\mathrm{Sel}^S(\mathcal{L}, A_n)$, we write $M^\vee$ for the
Pontrjagin dual $\mathrm{Hom}_{\BZ_p}(M, \BQ_p/\BZ_p)$ of $M$. Then
$\mathrm{Sel}^S(\mathcal{L}, A_{\rho})^\vee$ and
$\mathrm{Sel}^S(\mathcal{L}, A_n)^\vee$ are
$\mathcal{O}[[\Gamma]]$-modules.

For $\Lambda=\mathcal{O}_\mathfrak{P}[[\Gamma_K]],
\mathcal{O}_\mathfrak{P}[[\Gamma^-_K]],
\mathcal{O}_\mathfrak{P}[[\Gamma^-_J]] $ or
$\mathcal{O}_\mathfrak{P}[[\Gamma^-_{\mathfrak{p}}]]$, and for a
$\Lambda$-module $M$, we write $\mathrm{char}_\Lambda M$ for the
characteristic ideal of $M$. The reader may consult \cite[Section
3.1.5]{S-U} for the definition of characteristic ideals. When
$\Lambda$ is clear to us, we also write $\mathrm{char} M$ for
$\mathrm{char}_\Lambda M$. If $\mathfrak{a}$ is an ideal of
$\Lambda$, we write $M/\mathfrak{a}$ for the
$\Lambda/\mathfrak{a}$-module $M/\mathfrak{a}M$. We write
$\Lambda_{E}$ for
$\Lambda\otimes_{\mathcal{O}_{\mathfrak{P}}}E_{\mathfrak{P}}$.

We have the following direct consequences of Lemma
\ref{lem:control}.

\begin{cor} For $\mathcal{L}$ and $\Gamma=\mathrm{Gal}(\mathcal{L}/K)$ as above, we have
$$  \mathrm{Sel}^S(\mathcal{L}, A_{\rho}) = \mathrm{Sel}^S(\mathcal{K}_\infty, A_{\rho})^\Gamma
$$ and $$  \mathrm{Sel}^S(\mathcal{L}, A_n) = \mathrm{Sel}^S(\mathcal{K}_\infty,
A_n)^\Gamma.
$$
\end{cor}

Let $\mathfrak{a}_\mathcal{L}$ be the kernel of the natural
homomorphism $$\mathcal{O}_\mathfrak{P}[[\Gamma_{K}]]\rightarrow
\mathcal{O}_\mathfrak{P}[[\mathrm{Gal}(\mathcal{L}/K)]].$$ Let
$\mathfrak{b}_\mathcal{L}$ be the kernel of the natural homomorphism
$$\mathcal{O}_\mathfrak{P}[[\mathrm{Gal}(\mathcal{L}/K)]]\rightarrow
\mathcal{O}_\mathfrak{P}[[\Gamma^-_\mathfrak{p}]].$$

\begin{cor}\label{cor:char-sel} For $\mathcal{L}$ and $\Gamma=\mathrm{Gal}(\mathcal{L}/K)$ as
above, we have
$$  \mathrm{Sel}^S(\mathcal{L}, A_{\rho})^\vee = \mathrm{Sel}^S(\mathcal{K}_\infty,
A_{\rho})^\vee / \mathfrak{a}_\mathcal{L}
$$ and $$  \mathrm{Sel}^S(\mathcal{L}, A_n)^\vee = \mathrm{Sel}^S(\mathcal{K}_\infty,
A_n)^\vee  / \mathfrak{a}_\mathcal{L}.
$$
\end{cor}

\begin{prop}\label{prop:S-U} For $\mathcal{L}$ as above, we have $$\mathrm{char} \ \mathrm{Sel}
(\mathcal{L}, A_\rho)^\vee \subseteq \mathrm{char}\ \mathrm{Sel}
(\mathcal{K}_\infty, A_\rho)^\vee \ \mathrm{mod}\
\mathfrak{a}_\mathcal{L} $$ and  $$\mathrm{char} \ \mathrm{Sel}
(K_\infty, A_\rho)^\vee \subseteq \mathrm{char}\ \mathrm{Sel}
(\mathcal{L}, A_\rho)^\vee \ \mathrm{mod}\ \mathfrak{b}_\mathcal{L}
.$$
\end{prop}
\begin{proof} This follows from Corollary \ref{cor:char-sel} and \cite[Corollary
3.8(ii)]{S-U}. Note that
$\mathcal{O}_\mathfrak{P}[[\mathrm{Gal}(\mathcal{L}/K)]]$ and
$\mathcal{O}_\mathfrak{P}[[\Gamma_K]]$  are unique factorization
domains.
\end{proof}

\subsection{Iwasawa main conjecture}

One attaches to $f$ and the extension $\mathcal{K}_\infty$ a
$p$-adic L-function $$(L_p^S(\mathcal{K}_\infty,
f))\in\mathcal{O}_\mathfrak{P}[[\Gamma_K]]$$ (see \cite[Theorem
82]{Wan}). For $\mathcal{L}=\mathcal{K}^-_\infty$, $K^-_J$ or
$K_\infty$, by specializing we has $(L_p^S(\mathcal{L}, f))\in
\mathcal{O}_\mathfrak{P}[[\Gal_{\mathcal{L}/K}]]$. In fact, the
definition of $(L_p^S(\mathcal{K}_\infty,
f))\in\mathcal{O}_\mathfrak{P}[[\Gamma_K]]$ needs a choice of the
period of $f$. As we will only use Wan's $\BQ$-version result and
need the integral version result in \cite{Xie}, we can choose the
period such that $(L_p(K_\infty, f))$ coincides with $p$-adic
$L$-function for the $K_\infty$-extension used in \cite{Xie}.

For two ideals $I$ and $J$ of
$\Lambda=\mathcal{O}_\mathfrak{P}[[\Gamma_K]]$, if $\Lambda_E\cdot I
=\Lambda_E\cdot J$, we write $I\doteq J$.

\begin{thm}\label{thm:wan} $($\cite[Theorem 3]{Wan}$)$ Suppose $[F:\BQ]$ is even. If $S$ is a finite set of finite places of $F$ that contains all primes dividing
$\mathfrak{n}$, then $\mathrm{Sel}^S(\mathcal{K}_\infty, A_{\rho})$
is a cofinitely generated cotorsion
$\mathcal{O}_\mathfrak{P}[[\Gamma_{\mathcal{K}}]]$-module and
\begin{equation}\label{eq:wan-2} \mathrm{char}\ \mathrm{Sel}^S
(\mathcal{K}_\infty, A_{\rho})^\vee \subseteq
(L_p^S(\mathcal{K}_\infty, f)) \end{equation} in
$\mathcal{O}_\mathfrak{P}[[\Gamma_{\mathcal{K}}]]_E$.
\end{thm}

\begin{rem}
In the case when $[F:\BQ_p]$ is odd,  \cite[Theorem 3]{Wan} needs
the condition that Ihara Lemma for Shimura curves holds. \end{rem}

\begin{cor} Suppose $[F:\BQ]$ is even. Then $\mathrm{Sel}(\mathcal{K}_\infty, A_\rho)$ is a cofinitely
generated cotorsion $\mathcal{O}_\mathfrak{P}[[\Gamma_{K}]]$-module
and
\begin{equation}\label{eq:wan-2} \mathrm{char}\ \mathrm{Sel}
(\mathcal{K}_\infty, A_{\rho})^\vee \subseteq (L_p
(\mathcal{K}_\infty, f)) \end{equation} in
$\mathcal{O}_\mathfrak{P}[[\Gamma_{K}]]_E$.
\end{cor}
\begin{proof} We take $S$ to be the set of primes dividing $\mathfrak{n}^-$. Then we have an exact sequence
$$ \xymatrix{ 0 \ar[r] & \mathrm{Sel}
(\mathcal{K}_\infty, A_{\rho})\ar[r] & \mathrm{Sel}^S
(\mathcal{K}_\infty, A_{\rho})\ar[r] & \prod_{\mathfrak{l}|
\mathfrak{n}^- } H^1((\mathcal{K}_\infty)_\mathfrak{l},
F^-_\mathfrak{l}A_\rho).}
$$ Being a submodule of a cofinitely
generated cotorsion $\mathcal{O}_\mathfrak{P}[[\Gamma_{K}]]$-module,
$\mathrm{Sel} (\mathcal{K}_\infty, A_{\rho})$ is a cofinitely
generated cotorsion $\mathcal{O}_\mathfrak{P}[[\Gamma_{K}]]$-module
itself.

Applying \cite[Chapter XIX Proposition 2.7]{Lang} we obtain
\begin{equation}\label{eq:fitting-1} \mathrm{Fitt}(\mathrm{Sel}(\mathcal{K}_\infty,
A_{\rho})^\vee) \cdot \mathrm{Fitt}(\prod_{\mathfrak{l}|
\mathfrak{n}^- }H^1((\mathcal{K}_\infty)_\mathfrak{l},
F^-_\mathfrak{l}A_\rho)^\vee) \subseteq \mathrm{Fitt}(\mathrm{Sel}^S
(\mathcal{K}_\infty, A_{\rho})^\vee).  \end{equation} The reader may compare it with
\cite[Corollary 3.12]{S-U}. Here, for a finite generated
$\mathcal{O}_\mathfrak{P}[[\Gamma_{K}]]$-module $M$, we use
$\mathrm{Fitt}(M)$ to denote its Fitting ideal. Combining Lemma \ref{lem:fitt} below and
 Theorem \ref{thm:wan} we obtain
\begin{equation}\label{eq:fitting-2}\mathrm{Fitt}\ \mathrm{Sel}
(\mathcal{K}_\infty, A_{\rho})^\vee \subseteq (L_p
(\mathcal{K}_\infty, f)).\end{equation} As the characteristic ideal is the smallest divisorial ideal containing the Fitting ideal, (\ref{eq:fitting-1}) and (\ref{eq:fitting-2}) yield the conclusion.
\end{proof}

\begin{lem}\label{lem:fitt} Let $S$ be the set of primes dividing $\mathfrak{n}^-$. Then
$$ \mathrm{Fitt}(\prod_{\mathfrak{l}|
\mathfrak{n}^- }H^1((\mathcal{K}_\infty)_\mathfrak{l},
F^-_\mathfrak{l}A_\rho)^\vee) =\left(\frac{L^S_p (\mathcal{K}_\infty,
f)}{L_p (\mathcal{K}_\infty, f)}\right).
$$
\end{lem} This should be well known.
However, the authors do not have a reference about it.
\begin{proof}
By definition (see Section \ref{ss:selmer}) $H^{1}((\mathcal{K}_{\infty})_{\mathfrak{l}},F_{\mathfrak{l}}^{-}A_{\rho})$ is the inductive limit $\lim\limits_{\overrightarrow{\;\;L\;\;}} H^{1}(L_{\mathfrak{l}},F_{\mathfrak{l}}^{-}A_{\rho}),$ where $L$ runs through all finite extensions of $K$ contained in $\mathcal{K}_\infty$. For each $L$,  $H^{1}(L_{\mathfrak{l}},F_{\mathfrak{l}}^{-}A_{\rho})$ is isomorphic to
$\bigoplus_{\mathfrak{L}|\mathfrak{l}} H^1(G_{L_\mathfrak{L}}, F_{\mathfrak{l}}^{-}A_{\rho})$ where $\mathfrak{L}$ runs over all places of $L$ above $\mathfrak{l}$,
and sits in the following exact sequence
{\small $$ \bigoplus_{\mathfrak{L}|\mathfrak{l}} H^1(G_{L_\mathfrak{L}}/I_{L_\mathfrak{L}}, (F_{\mathfrak{l}}^{-}A_{\rho})^{I_{L_\mathfrak{L}}})\rightarrowtail \bigoplus_{\mathfrak{L}|\mathfrak{l}} H^1(G_{L_\mathfrak{L}}, F_{\mathfrak{l}}^{-}A_{\rho})\twoheadrightarrow \bigoplus_{\mathfrak{L}|\mathfrak{l}} H^1(I_{L_\mathfrak{L}}, F_{\mathfrak{l}}^{-}A_{\rho})^{G_{L_\mathfrak{L}}/I_{L_\mathfrak{L}}} . $$}
As $G_{L_{\mathfrak{L}}}$ acts trivially on $F_\mathfrak{l}^-A_\rho$, we have
$$ H^1(G_{L_\mathfrak{L}}/I_{L_\mathfrak{L}}, (F_{\mathfrak{l}}^{-}A_{\rho})^{I_{L_\mathfrak{L}}}) \cong F_{\mathfrak{l}}^{-}A_{\rho}. $$ If $L'$ is a finite extension of $L$ contained in $\mathcal{K}_\infty$, and $\mathfrak{L}'$ is a place of $L'$ above $\mathfrak{L}$, we have the following commutative diagram
\[ \xymatrix{ H^1(G_{L_\mathfrak{L}}/I_{L_\mathfrak{L}}, (F_{\mathfrak{l}}^{-}A_{\rho})^{I_{L_\mathfrak{L}}})\ar[rr]^{\hskip45pt \simeq}\ar[d] && F_{\mathfrak{l}}^{-}A_{\rho} \ar[d]^{\times[G_{L_\mathfrak{L}}/I_{L_\mathfrak{L}}:G_{L'_{\mathfrak{L}'}}/I_{L'_{\mathfrak{L}'}}]} \\ H^1(G_{L'_{\mathfrak{L}'}}/I_{L'_{\mathfrak{L}'}}, (F_{\mathfrak{l}}^{-}A_{\rho})^{I_{L'_{\mathfrak{L}'}}})\ar[rr]^{\hskip45pt\simeq} && F_{\mathfrak{l}}^{-}A_{\rho},  } \] where the left vertical arrow is the restriction map. Since both $L_{\mathfrak{L}}$ and $L'_{\mathfrak{L}'}$ are unramified over $K_\mathfrak{l}$, $$[G_{L_\mathfrak{L}}/I_{L_\mathfrak{L}}:G_{L'_{\mathfrak{L}'}}/I_{L'_{\mathfrak{L}'}}]
=|\mathrm{Gal}(L'_{\mathfrak{L}'}/L_{\mathfrak{L}})|=[L'_{\mathfrak{L}'}:L_{\mathfrak{L}}]$$ is a power of $p$. Thus taking limit we have $$\lim\limits_{\overrightarrow{\;\;L\;\;}}\bigoplus_{\mathfrak{L}: \text{places of } L \text{ above }\mathfrak{l}} H^1(G_{L_\mathfrak{L}}/I_{L_\mathfrak{L}}, (F_{\mathfrak{l}}^{-}A_{\rho})^{I_{L_\mathfrak{L}}})=0.$$

Hence, we have an isomorphism  $$ H^{1}((\mathcal{K}_{\infty})_{\mathfrak{l}},F_{\mathfrak{l}}^{-}A_{\rho}) \cong \lim\limits_{\overrightarrow{\;\;L\;\;}}\bigoplus_{\mathfrak{L}: \text{places of }  L \text{ above }\mathfrak{l}} H^1(I_{L_\mathfrak{L}}, F_{\mathfrak{l}}^{-}A_{\rho} )^{G_{L_\mathfrak{L}}/I_{L_\mathfrak{L}}}  . $$ As a $\Gal(L/K)$-module, $$\bigoplus_{\mathfrak{L}: \text{places of }  L \text{ above }\mathfrak{l}} H^1(I_{L_\mathfrak{L}}, F_{\mathfrak{l}}^{-}A_{\rho} )^{G_{L_\mathfrak{L}}/I_{L_\mathfrak{L}}}$$ is isomorphic to $$\mathrm{Ind}^{\Gal(L/K)}_{\Gal({L_{\mathfrak{L}}/K_\mathfrak{l}})} H^1(I_{L_\mathfrak{L}}, F_{\mathfrak{l}}^{-}A_{\rho} )^{ G_{L_\mathfrak{L}}/I_{L_\mathfrak{L}} }$$ the $\Gal({L/K})$-module induced from the $\Gal({L_{\mathfrak{L}}/K_\mathfrak{l}})$-module $H^1(I_{L_\mathfrak{L}},  F_{\mathfrak{l}}^{-}A_{\rho} )^{ G_{L_\mathfrak{L}}/I_{L_\mathfrak{L}} } $.
Let $\bar{\mathfrak{l}}$ be a place of $\mathcal{K}_\infty$ above $\mathfrak{l}$. As a  $\Gamma_K=\Gal({\mathcal{K}_\infty /K})$-module, the inductive limit $$\lim\limits_{\overrightarrow{\;\;L\;\;}}\bigoplus_{\mathfrak{L}: \text{places of }  L \text{ above }\mathfrak{l}} H^1(I_{L_\mathfrak{L}}, F_{\mathfrak{l}}^{-}A_{\rho} )^{G_{L_\mathfrak{L}}/I_{L_\mathfrak{L}}}  $$ is isomorphic to $$\mathrm{Ind}^{\Gal({\mathcal{K}_\infty/K})}_{\Gal({(\mathcal{K}_\infty)_{\bar{\mathfrak{l}}}/K_\mathfrak{l}})} H^1(I_{\bar{\mathfrak{l}}},  F_{\mathfrak{l}}^{-}A_{\rho} )^{ G_{(\mathcal{K}_\infty)_{\bar{\mathfrak{l}}}}/I_{\bar{\mathfrak{l}}} }.$$
where $I_{\bar{\mathfrak{l}}}$ is the inertia subgroup of $G_{(\mathcal{K}_\infty)_{\bar{\mathfrak{l}}}}$.

The group $I_{\bar{\mathfrak{l}}}$ has a unique subgroup $J_{\bar{\mathfrak{l}}}$ such that $I_{\bar{\mathfrak{l}}}/J_{\bar{\mathfrak{l}}}\simeq \mathbb{Z}_{p}(1)$ with respect to the action of $\mathrm{Gal}( K_{\mathfrak{l} }^{\mathrm{un}}| K_{\mathfrak{l}} )$. We have the inflation-restriction sequence $$0\longrightarrow H^{1}(I_{\bar{\mathfrak{l}}}/J_{\bar{\mathfrak{l}}},
(F_{\mathfrak{l}}^{-}A_{\rho})^{J_{\bar{\mathfrak{l}}}})\longrightarrow H^{1}(I_{\bar{\mathfrak{l}}},F_{\mathfrak{l}}^{-}A_{\rho})\longrightarrow H^{1}(J_{\bar{\mathfrak{l}}},F_{\mathfrak{l}}^{-}A_{\rho})^{I_{\bar{\mathfrak{l}}}/J_{\bar{\mathfrak{l}}}}.$$ As $J_{\bar{\mathfrak{l}}}$ has profinite order prime to $p$ and acts trivially on $F_{\mathfrak{l}}^{-}A_{\rho}$, we have $$H^{1}(J_{\mathfrak{l}},F_{\mathfrak{l}}^{-}A_{\rho})\cong \mathrm{Hom}(J_{\mathfrak{l}},F_{\mathfrak{l}}^{-}A_{\rho})=0$$ and \begin{eqnarray*}H^1(I_{\bar{\mathfrak{l}}},  F_{\mathfrak{l}}^{-}A_{\rho} )^{ G_{(\mathcal{K}_\infty)_{\bar{\mathfrak{l}}}}/I_{\bar{\mathfrak{l}}} }&=&H^{1}(I_{\bar{\mathfrak{l}}},F_{\mathfrak{l}}^{-}A_{\rho})\cong H^{1}(I_{\bar{\mathfrak{l}}} /J_{\bar{\mathfrak{l}}},
(F_{\mathfrak{l}}^{-}A_{\rho})^{J_{\bar{\mathfrak{l}}}}) \\ &\cong&
\Hom_{\mathrm{cont}}(I_{\mathfrak{l}}/J_{\mathfrak{l}},F_{\mathfrak{l}}^{-}A_{\rho})\cong
\Hom_{\mathrm{cont}}(\mathbb{Z}_p(1),F_{\mathfrak{l}}^{-}A_{\rho})\end{eqnarray*} with respect to the action of $\Gal({ (\mathcal{K}_\infty)_{\bar{\mathfrak{l}}} /K_\mathfrak{l} })$.
 So $\Gal({(\mathcal{K}_\infty)_{\bar{\mathfrak{l}}}/K_\mathfrak{l}})$ acts on $$ H^1(I_{\bar{\mathfrak{l}}},  F_{\mathfrak{l}}^{-}A_{\rho} )^{ G_{(\mathcal{K}_\infty)_{\bar{\mathfrak{l}}}}/I_{\bar{\mathfrak{l}}} }\cong E_{\mathfrak{P}}/\mathcal{O}_\mathfrak{P} $$ via $\epsilon^{-1}$. Hence
  $$ H^1((\mathcal{K}_\infty)_\mathfrak{l},
F^-_\mathfrak{l}A_\rho)^\vee\cong \mathcal{O}_\mathfrak{P}[[\Gamma_K]]/(1-\mathrm{N}(\mathfrak{l})^{-1}\mathrm{Frob}_{\mathfrak{l}}) $$ and $$\mathrm{Fitt}(H^{1}((\mathcal{K}_{\infty})_{\mathfrak{l}},F_{\mathfrak{l}}^{-}A_{\rho})^{\vee})
=(1- \mathrm{N}(\mathfrak{l})^{-1}\mathrm{Frob}_{\mathfrak{l}}),$$ where $\mathrm{Frob}_{\mathfrak{l}}$ denotes the Frobenius element in $\mathrm{Gal}(\mathcal{K}_{\infty}/K)_{\mathfrak{l}}\cong \mathrm{Gal}((\mathcal{K}_{\infty})_{\bar{\mathfrak{l}}}/K_{\mathfrak{l}})$ that satisfies $\mathrm{Frob}_{\mathfrak{l}}(x)\equiv x^{\mathrm{N}(\mathfrak{l})}$ mod $\mathfrak{l}$. One  should note that $\epsilon(\mathrm{Frob}_{\mathfrak{l}})=\mathrm{N}(\mathfrak{l}).$

For each Hecke character $\chi$ of $\mathrm{Gal}(\mathcal{K}_{\infty}/K) $, we have $$L(f,\chi,s)=L^{S}(f,\chi,s)\cdot\prod_{\mathfrak{l}|\mathfrak{n}^{-}}
(1- \chi(\mathfrak{l})\mathrm{N}(\mathfrak{l})^{1-s})^{-1}.$$
By the interpolation formula of $p$-adic $L$-functions, we have $$\frac{L^{S}_{p}(\mathcal{K}_{\infty},f)}{L_{p}(\mathcal{K}_{\infty},f)}
=\prod_{\mathfrak{l}|\mathfrak{n}^{-}}(1-\mathrm{N}(\mathfrak{l})^{-1}\mathrm{Frob}_{\mathfrak{l}}),$$
as desired.
\end{proof}

\begin{thm}\label{thm:xie} $($\cite[Theorem 1.2]{Xie}$)$   $\mathrm{Sel}(K_\infty, A_{\rho})$ is a cofinitely generated cotorsion
$\mathcal{O}_\mathfrak{P}[[\Gamma^-_{\mathfrak{p}}]]$-module and $$
\mathrm{char}\ \mathrm{Sel} (K_\infty, A_{\rho})^\vee \supseteq
(L_p(K_\infty, f)).
$$
\end{thm}

\begin{lem}\label{lem:quad-ext} There exists a totally real quadratic extension $F'$ of $F$,
such that $f'$, the base change  of $f$ to $F'$, again satisfies
$(\mathrm{CR}^+)$, $(\mathfrak{n}^+\text{-}\mathrm{DT})$,
$(\mathrm{PO})$ and $(\mathrm{Fuji}1$-$4)$.
\end{lem}

Here, the meaning of our notation $F'$ is different from that in the
introduction. This makes no confusion, because in this section we
will not deal with $F'$ in our introduction.

\begin{proof} The representation $\bar{\rho}_f$ factors through
$\mathrm{Gal}(F_1/F)$ for a finite Galois extension $F_1$ of $F$.
Take a finite place $\mathfrak{q}\nmid p\mathfrak{n}$. For each
$\mathfrak{l}|p\mathfrak{n}$ take a unit $u_\mathfrak{l}$ of
$\mathcal{O}_{F_{\mathfrak{l}}}$ such that the image
$u_\mathfrak{l}$ in $k_{\mathfrak{l}}$ by the natural homomorphism
$$ \mathcal{O}_{F_\mathfrak{l}}\rightarrow
\mathcal{O}_{F_\mathfrak{l}}/\mathfrak{l}\cong k_\mathfrak{l} $$
generates the multiplicative group $k_\mathfrak{l}^\times$. Take a
uniformizing element $\omega_\mathfrak{q}$ of
$\mathcal{O}_{F_{\mathfrak{q}}}$. By the weak approximation theorem
there exists a totally positive element $a$ of $F$ such that
$a_{\mathfrak{l}}-u_{\mathfrak{l}} \in
\mathfrak{l}\mathcal{O}_{F_\mathfrak{l}}$ for each
$\mathfrak{l}|p\mathfrak{n}$, and
$a_\mathfrak{q}-\omega_\mathfrak{q}\in
\mathfrak{q}^2\mathcal{O}_{F_\mathfrak{q}}$. Put $F'=F(\sqrt{a})$.
Then each place dividing $p\mathfrak{n}$ splits in $F'$. As $a$ is
totally positive, $F'$ is totally real. Let $f'$ be the base change
of $f$ to $F'$.

Since $F_1$ is unramified at $\mathfrak{q}$ and $F'$ is ramified at
$\mathfrak{q}$, as extensions of $F$, $F'$ and $F_1$ are disjoint
from each other. As a consequence, $\bar{\rho}_{f'}  $ and
$\bar{\rho}_f  $ have the same image. So we have $(\mathrm{Fuji}3)$
for $f'$. Similarly, $\bar{\rho}_{f} |_{G_{F(\xi_p)}}$ and
$\bar{\rho}_{f'} |_{G_{F'(\xi_p)}}$ have the same image, so
$(\mathrm{Fuji}4)$ holds for $ {f'}$. By the same reason we have
$(\mathrm{CR}^+$$2)$ for $ {f'}$.

Put $K'=KF'$. Let $\mathfrak{n}'$ be the conductor of $f'$. With
$K'$ in place of $K$, we have
$\mathfrak{n}'=\mathfrak{n}'^+\mathfrak{n}'^-$, where
\(\mathfrak{n}'^{+}\) (resp. \(\mathfrak{n}'^{-}\)) is only
divisible by primes that are split (resp. inert or ramified) in
\(K'\). As each place dividing $\mathfrak{n}$ splits in $F'$, we
have $\mathfrak{n}'^+=\mathfrak{n}^+\mathcal{O}_{F'}$ and
$\mathfrak{n}'^-=\mathfrak{n}^-\mathcal{O}_{F'}$. Now it is easy to
check that $f'$ satisfies the remaining conditions.
\end{proof}

\begin{thm}\label{thm:Iw-main} Assume $f$ satisfies
$(\mathrm{CR}^+)$, $(\mathfrak{n}^+\text{-}\mathrm{DT})$,
$(\mathrm{PO})$ and $(\mathrm{Fuji}1$-$4)$.
\begin{enumerate}
\item\label{it:Iw-main-a}
We have \begin{equation}\label{eq:char} \mathrm{char}\ \mathrm{Sel}
(K_\infty, A_{\rho})^\vee = (L_p(K_\infty, f)). \end{equation}
\item\label{it:Iw-main-b} Let $\mathcal{L}$ be $\mathcal{K}_\infty,
\mathcal{K}^-_\infty$ or $K^-_J$. If $[F:\BQ]$ is even, then
$$ \mathrm{char}\ \mathrm{Sel}(\mathcal{L}, A_\rho)^\vee = (L_p(\mathcal{L}, f))
$$ in $\mathcal{O}_\mathfrak{P}[[\mathrm{Gal}(\mathcal{L}/K)]]_E$.
\end{enumerate}
\end{thm}
\begin{proof} Let $F'$ be a totally real quadratic extension of $F$, as in Lemma \ref{lem:quad-ext}. Then $[F':\BQ]$ is even.
Let $\chi_{F'/F}$ be the quadratic character corresponding to the
extension $F'/F$. Put $K'=F'K$. Let $\mathcal{K}'^-_\infty$ be the
maximal abelian anticyclotomic $\BZ_p$-extension of $K'$ unramified
outside $p$. Put $\mathcal{K}'^+_\infty=F_\infty K'$,
$\mathcal{K}'_\infty=\mathcal{K}'^+_\infty\mathcal{K}'^-_\infty$ and
$\Gamma_{K'}=\mathrm{Gal}(\mathcal{K}'_\infty/K')$. Then
$K'\mathcal{K}_\infty=F'\mathcal{K}_\infty$ is contained in
$\mathcal{K}'_\infty$.

Let $f'$ be the base change of $f$ to $F'$. Let $\rho'$ be the
Galois representation attached to $f'$ and we have
$A_{\rho'}=A_{\rho}|_{G_{F'}}$ according to base change. By Theorem
\ref{thm:wan},
$$ \mathrm{char}\ \mathrm{Sel}
(\mathcal{K}'_\infty, A_{\rho'})^\vee \subseteq
(L_p(\mathcal{K}'_\infty, f'))
$$in $\mathcal{O}_\mathfrak{P}[[\Gamma_{K'}]]_E$. So by Proposition \ref{prop:S-U}, $$ \mathrm{char}\ \mathrm{Sel} (K'\mathcal{K}_\infty,
A_{\rho'})^\vee \subseteq (L_p(K'\mathcal{K}_\infty, f')).
$$ in $\mathcal{O}_\mathfrak{P}[[\mathrm{Gal}(K'\mathcal{K}_\infty/K')]]_E\cong \mathcal{O}_\mathfrak{P}[[\Gamma_{K}]]_E$. But we have $$ \mathrm{Sel} (K'\mathcal{K}_\infty,
A_{\rho'})=\mathrm{Sel} (\mathcal{K}_\infty, A_\rho)\bigoplus
\mathrm{Sel} (\mathcal{K}_\infty, A_{\rho\otimes \chi_{F'/F}}) $$
and
$$ L_p(K'\mathcal{K}_\infty, f')\doteq L_p(\mathcal{K}_\infty,
f)L_p(\mathcal{K}_\infty, f\otimes \chi_{F'/F}).$$ Then $$
\mathrm{char}\ \mathrm{Sel} (\mathcal{K}_\infty, A_{\rho})^\vee\cdot
\mathrm{char}\ \mathrm{Sel} (\mathcal{K}_\infty, A_{\rho\otimes
\chi_{F'/F}})^\vee \subseteq (L_p(\mathcal{K}_\infty,
f)L_p(\mathcal{K}_\infty, f\otimes \chi_{F'/F}) )
$$ in $\mathcal{O}_\mathfrak{P} [[\Gamma_K]]_E
$. Hence, by Proposition \ref{prop:S-U},
\begin{equation}\label{eq:wan-111} \mathrm{char}\ \mathrm{Sel}
(K_\infty, A_\rho)^\vee\cdot \mathrm{char}\ \mathrm{Sel} (K_\infty,
A_{\rho\otimes \chi_{F'/F}})^\vee \subseteq (L_p(K_\infty,
f)L_p(K_\infty, f\otimes \chi_{F'/F}) ) \end{equation} in
$\mathcal{O}_\mathfrak{P}[[\Gamma^-_{\mathfrak{p}}]]_E$.

By Theorem \ref{thm:xie}
$$ \mathrm{char}\ \mathrm{Sel} (K_\infty, A_{\rho})^\vee  \supseteq
(L_p(K_\infty, f)).
$$ Note that $\rho_{f\otimes
\chi_{F'/F}}$ again satisfies $(\mathrm{CR}^+)$,
$(\mathfrak{n}^+\text{-}\mathrm{DT})$, $(\mathrm{PO})$ and
$(\mathrm{Fuji}1$-$4)$. By Theorem \ref{thm:xie} we also have $$
\mathrm{char}\ \mathrm{Sel} (K_\infty, A_{\rho\otimes
\chi_{F'/F}})^\vee \supseteq (L_p(K_\infty, f\otimes \chi_{F'/F}) )
.$$ In particular, these relations hold in
$\mathcal{O}_\mathfrak{P}[[\Gamma^-_{\mathfrak{p}}]]_E$. If either
$$ \mathrm{char}\ \mathrm{Sel} (K_\infty, A_{\rho})^\vee \supsetneq
(L_p(K_\infty, f))  $$ or  $$  \mathrm{char}\ \mathrm{Sel}
(K_\infty, A_{\rho\otimes \chi_{F'/F}})^\vee \supsetneq (
L_p(K_\infty, f\otimes \chi_{F'/F}) )  $$ in
$\mathcal{O}_\mathfrak{P}[[\Gamma^-_\mathfrak{p}]]_E$, then we have
$$ \mathrm{char}\ \mathrm{Sel} (K_\infty, A_{\rho})^\vee\cdot \mathrm{char}\
\mathrm{Sel} (K_\infty, A_{\rho\otimes \chi_{F'/F}})^\vee \supsetneq
(L_p(K_\infty, f)L_p(K_\infty, f\otimes \chi_{F'/F}) )
$$ in $\mathcal{O}_\mathfrak{P}[[\Gamma^-_\mathfrak{p}]]_E$, which contradicts (\ref{eq:wan-111}). Hence (\ref{eq:char})
holds in $\mathcal{O}_\mathfrak{P}[[\Gamma^-_\mathfrak{p}]]_E$.
According to \cite[Theorem 6.9]{Hung}, the (analytic)
$\mu$-invariant of $L_p(K_\infty, f)$ is zero. So by Theorem
\ref{thm:xie}, (\ref{eq:char}) already holds in
$\mathcal{O}_\mathfrak{P}[[\Gamma^-_\mathfrak{p}]]$. This proves
(\ref{it:Iw-main-a}).

For (\ref{it:Iw-main-b}) without loss of generality we may assume
$\mathfrak{p}\in J$. By Proposition \ref{prop:S-U} and Theorem
\ref{thm:wan} we obtain $$  \mathrm{char}\ \mathrm{Sel}
(\mathcal{L}, A_{\rho})^\vee \subseteq \mathrm{char}\ \mathrm{Sel}
(\mathcal{K}_\infty, A_{\rho})^\vee \mathrm{mod} \
\mathfrak{a}_\mathcal{L}\subseteq (L_p(\mathcal{L}, f))$$ in
$\mathcal{O}_\mathfrak{P}[[\mathrm{Gal}(\mathcal{L}/K)]]_E$.  By
(\ref{it:Iw-main-a}) and Proposition \ref{prop:S-U} we have
$$\mathrm{char}\ \mathrm{Sel}(\mathcal{L}, A_\rho)^\vee \
\mathrm{mod} \ \mathfrak{b}_\mathcal{L} = (L_p(\mathcal{L}, f)) \
\mathrm{mod} \ \mathfrak{b}_\mathcal{L} = (L_p(K_\infty, f))
$$  in $\mathcal{O}_\mathfrak{P}[[\Gamma^-_\mathfrak{p}]]_E$. Now
(\ref{it:Iw-main-b}) follows from \cite[Lemma 3.2]{S-U}.
\end{proof}

\section{Period and $\mu$-invariant}\label{sec:e}

Theorem \ref{thm:main} is a consequence of the following theorem.

\begin{thm}\label{thm:main-stronger} Let $F$ be a totally real field,
and $F'$ be a real quadratic extension of $F$. Assume that $p$ is
unramified in both $F$ and $F'$. Let $f$ be a Hilbert newform of
parallel even weight $k$ over $F$ with trivial central character and
level $\mathfrak{n}_f$. Let $f'$ be the base change of $f$ to $F'$.
Let $\mathfrak{n}_{f'}$ be the level of $f'$.

We assume that $\mathfrak{n}_f$ is prime to $p$ and $f$ is ordinary
at each place above $p$. When $[F:\mathbb{Q}]$ is odd, we assume
that there exists at least one prime $\mathfrak{q}$ such that
$\mathfrak{q}|| \mathfrak{n}_f$, $\mathfrak{q}$ is split in $F'$,
and $\bar{\rho}_{f}$ is ramified at $\mathfrak{q}$. Suppose that, if
$\mathfrak{l}'||\mathfrak{n}_{f'}$ and
$\mathrm{N}(\mathfrak{l}')\equiv 1 \ (\mathrm{mod} \ p)$, then
$\bar{\rho}_{f'}$ is ramified at $\mathfrak{l}'$.

Further suppose that $(\mathrm{CR}^+$$1,2,4)$, $(\mathrm{PO})$ and
$(\mathrm{Fuji1}$-$4)$ hold for $f$ and $f'$.Then $$
\Omega_{f'}^\mathrm{can}=(\Omega_f^\mathrm{can})^2 $$ up to a
$p$-adic unit.
\end{thm}


Fix a place $\mathfrak{p}$ of $F$ above $p$. We choose an imaginary
quadratic extension $K$ of $F$ that is split at $\mathfrak{p}$ and
each place dividing $\mathfrak{n}$ in the case when $[F:\mathbb{Q}]$
is even, and that is split at $\mathfrak{p}$ and at each place
dividing $\frac{\mathfrak{n}}{\mathfrak{q}}$ and is inertia at
$\mathfrak{q}$ when $[F:\mathbb{Q}]$ is odd. Then
$\mathfrak{n}^-=\mathcal{O}_F$ in the case when $[F:\mathbb{Q}]$ is
even, and $\mathfrak{n}^-=\mathfrak{q}$ in the case when
$[F:\mathbb{Q}]$ is odd. Thus $(\mathrm{CR}^+3)$ and
($\mathfrak{n}^{+}$-DT) hold for $f$ and $f'$.

Denote $KF'$ by $K'$.

By \cite[Theorem 9.1 and Lemma 9.4]{Wang}
under (CR$^+$) and ($\mathfrak{n}^{+}$-DT), Gross period and the
canonical period only differ a $p$-adic unit\footnote{Though this conclusion is not stated in \cite{Wang}, repeating the argument in  \cite[Section 6]{CH15}, one can deduce it from \cite[Theorem 9.1 and Lemma 9.4]{Wang}.}. So the two kinds of
$p$-adic $L$-functions in \cite{Wan} and \cite{Xie} coincide.

\begin{thm}\label{cor:lv} We have $$
(L_{p}(K'K_{\infty},f'))=(L_{p}(K_{\infty},f))(L_{p}(K_{\infty},f\otimes\chi_{F'/F}))
$$ in $\mathcal{O}_{\mathfrak{P}}[[\Gamma_{\mathfrak{p}}^{-}]]. $
\end{thm}
\begin{proof}
We fix a place \(\mathfrak{p}'\) of \(F'\) above \(\mathfrak{p}\),
and  write \(K_{\infty}'\) for \({{K'}^{ \
-}_{\{\mathfrak{p}'\},\infty}}\).
    We have $$ L_{p}(K'K_{\infty},f')\doteq L_{p}(K_{\infty},f)L_{p}(K_{\infty},f\otimes\chi_{F'/F}) $$
    and $$ \mathrm{Sel}(K'K_{\infty},A_{\rho})=\mathrm{Sel}(K_{\infty},A_{\rho})\bigoplus\mathrm{Sel}(K_{\infty},A_{\rho\otimes\chi_{F'/F}}).$$
    By Theorem \ref{thm:Iw-main} (\ref{it:Iw-main-a})  we have $$ (L_{p}(K_{\infty},f))=\mathrm{char}\ \mathrm{Sel}(K_{\infty},A_{\rho})^{\vee},$$
    $$  (L_{p}(K_{\infty},f\otimes\chi_{F'/F}))=\mathrm{char}\ \mathrm{Sel}(K_{\infty},A_{\rho\otimes\chi_{F'/F}})^{\vee}    $$
    and $$  (L_{p}(K_{\infty}',f'))=\mathrm{char}\ \mathrm{Sel}(K_{\infty}',A_{\rho})^{\vee}.
    $$
    Since $K'K_{\infty} $ is a subextension of $K'$ contained in $ K_{\infty}'$, by the same method as in the proof of Proposition \ref{prop:S-U},
    we obtain $$\mathrm{char}\ \mathrm{Sel}(K'K_{\infty},A_{\rho})^{\vee}\subseteq \mathrm{char}\ \mathrm{Sel}(K_{\infty}',A_{\rho})^{\vee} \ \mathrm{mod}\ \mathfrak{c} $$
    where $ \mathfrak{c}$ is the kernel of $$\mathcal{O}_\mathfrak{P}[[\mathrm{Gal}(K_{\infty}'/K')]]\rightarrow
\mathcal{O}_\mathfrak{P}[[\mathrm{Gal}(K'K_{\infty}/K')]]\cong\mathcal{O}_\mathfrak{P}[[\Gamma_{\mathfrak{p}}^{-}]].$$
Thus we have $$ \mathrm{char}\
\mathrm{Sel}(K'K_{\infty},A_{\rho})^{\vee}=(L_{p}(K_{\infty},f))\cdot
(L_{p}(K_{\infty},f\otimes\chi_{F'/F}))\subseteq
(L_{p}(K'K_{\infty},f')) $$ in
$\mathcal{O}_{\mathfrak{P}}[[\Gamma_{\mathfrak{p}}^{-}]].$ According
to \cite[Theorem 6.9]{Hung},  $ \mu(L_{p}(K_{\infty},f))=0$ and $
\mu(L_{p}(K_{\infty},f\otimes\chi_{F'/F}))=0 $. So we deduce that
$$(L_{p}(K'K_{\infty},f'))=(L_{p}(K_{\infty},f))(L_{p}(K_{\infty},f\otimes\chi_{F'/F})),$$ as desired.
\end{proof}

{\noindent \it Proof of Theorem \ref{thm:main-stronger}.} Write
$\mathfrak{n}^+\mathcal{O}_K=\mathfrak{N}^+\overline{\mathfrak{N}^+}$.
    We fix a place $\mathfrak{p}$ of $F$ above $p$ and a place $\mathfrak{p}'$ of $F'$ above $\mathfrak{p}$.
    For the $p$-adic anticyclotomic $L$-function $L_{p}(K_{\infty},f)$, one has the interpolation formula in \cite[Theorem 5.6]{Hung}.
    For each $p$-adic character $\widehat{\nu}$ of $\mathrm{Gal}(K_{\infty}/K)$ of conductor $\mathfrak{p}^{s}$, we have
    $$\widehat{\nu}(L_{p}(K_{\infty},f))=\frac{L(\frac{1}{2},\pi_{K}\otimes \nu)}
    {\Omega_{f}^{\mathrm{can}}}\cdot \nu(\mathfrak{N}^{+})\cdot (N_{F/\mathbb{Q}}\mathfrak{p})^{sk}\cdot
     e_{\mathfrak{p}}(f, \nu)^{2-\mathrm{ord}_{\mathfrak{p}}\mathfrak{n}}\cdot C $$
    where $C$ is a $p$-adic unit and $\Omega_{f}^{\mathrm{can}}$ is the canonical period of
    $f$. We refer to \cite[Section 5]{Hung} for the relation between $\nu$ and
    $\widehat{\nu}$.
Here the branch character $\chi_t$ in \cite[Theorem 5.6]{Hung} is
take to be the trivial one.

Let $\widehat{\nu}'$ denote the base change of Hecke character
$\widehat{\nu}$ over $K'$.

    Case 1.  $\mathfrak{p}$ is split in
    $F'$. Write $\mathfrak{p}\mathcal{O}_{F'}=\mathfrak{p}'\bar{\mathfrak{p}'}$. As $K$ is split at $\mathfrak{p}$, we have
    $\mathfrak{p}\mathcal{O}_K=\mathfrak{P} \overline{\mathfrak{P} }$.
    Similarly, we have $\mathfrak{p}'\mathcal{O}_{K'}=\mathfrak{P}' \overline{\mathfrak{P}'
    }$. According to the definition of
    $e_{\mathfrak{p}}(f, \nu)$ in \cite[Theorem 5.6]{Hung}, if $s=0$ we have
    $$e_{\mathfrak{p}}(f, \nu)=
    (1-\alpha_{\mathfrak{p}}^{-1} \nu(\mathfrak{P} ))(1-\alpha_{\mathfrak{p}}^{-1} \nu(\overline{\mathfrak{P} })),$$
    $$e_{\mathfrak{p}}(f\otimes\chi_{F'/F}, \nu)=
    (1-\alpha_{\mathfrak{p}}^{-1}\chi_{F'/F}(\mathfrak{p}
    )^{-1} \nu(\mathfrak{P} ))(1-\alpha_{\mathfrak{p}}^{-1}\chi_{F'/F}(\mathfrak{p})^{-1} \nu(\overline{\mathfrak{P} })),$$
    $$e_{\mathfrak{p}}(f', \nu')=
    (1-\alpha_{\mathfrak{p}'}^{-1} \nu({N_{K'/K}\mathfrak{P}'} ))(1-\alpha_{\mathfrak{p}'}^{-1} \nu(N_{K'/K}\overline{{\mathfrak{P}'} }))$$
where
$\alpha_{\mathfrak{p}}=a_{\mathfrak{p}}N\mathfrak{p}^{\frac{2-k}{2}}$
(resp.
$\alpha_{\mathfrak{p}'}=a_{\mathfrak{p}'}N{\mathfrak{p}'}^{\frac{2-k}{2}}$)
and $a_{\mathfrak{p}}$ (resp. $a_{\mathfrak{p}'}$) is a root of the
Hecke polynomial of $f$ (resp. $f'$) and is a $p$-adic unit. Since
$\mathfrak{p}$ is split in $F'$, we have
$\chi_{F'/F}(\mathfrak{p})=1$ and
$\alpha_{\mathfrak{p}}=\alpha_{\mathfrak{p}'}$. So we have
$$ e_{\mathfrak{p}}(f, \nu)=e_{\mathfrak{p}}(f\otimes\chi_{F'/F}, \nu)=e_{\mathfrak{p}}(f', \nu)
.$$ Now we select an unramified character $\widehat{\nu}$ such that
$ \nu(\mathfrak{N}^{+})$ and $e_{\mathfrak{p}}(f, \nu)$ are both
$p$-adic units. Then we have
    $$\widehat{\nu}(L_{p}(K_{\infty},f))=\frac{L(\frac{1}{2},\pi_{K}\otimes \nu)}{\Omega_{f}^{\mathrm{can}}}\cdot C_{1},$$
    $$\widehat{\nu}(L_{p}(K_{\infty},f\otimes\chi_{F'/F}))=\frac{L(\frac{1}{2},\pi_{K}\otimes \nu\otimes\chi_{F'/F})}{\Omega_{f\otimes\chi_{F'/F}}^{\mathrm{can}}}\cdot C_{2}$$
    and
    $$\widehat{\nu}'(L_{p}(K'K_{\infty},f'))=\frac{L(\frac{1}{2},\pi_{K'}\otimes \nu')}{\Omega_{f'}^{\mathrm{can}}}\cdot C_{3}$$
    where $C_{1}$, $C_{2}$, $C_{3}$ are all $p$-adic units. Using Theorem \ref{cor:lv} and the relation
    $$L(\frac{1}{2},\pi_{K'}\otimes\nu')=L(\frac{1}{2},\pi_{K}\otimes\nu)\cdot L(\frac{1}{2},\pi_{K}\otimes\nu\otimes\chi_{F'/F}),$$ we obtain that $
  \Omega_{f'}^{\text{can}}$ and $ \Omega_{f}^{\text{can}}\cdot\Omega_{f\otimes\chi_{F'/F}}^{\text{can}}$ differ a $p$-adic unit.

 Case 2. $\mathfrak{p}$ is inert in $F'$. Write $\mathfrak{p}\mathcal{O}_{F'}=\mathfrak{p}'$. We select $\widehat{\nu}$ which satisfy $s>0$
 and such that $\nu(\mathfrak{N}^{+})$ is a $p$-adic unit. When $s>0$, by definition we have $$e_{\mathfrak{p}}(f, \nu)
 =e_{\mathfrak{p}}(f\otimes\chi_{F'/F}, \nu)=e_{\mathfrak{p}}(f', \nu)=1.$$ Then we have
  $$\widehat{\nu}(L_{p}(K_{\infty},f))=\frac{L(\frac{1}{2},\pi_{K}\otimes \nu)}{\Omega_{f}^{\mathrm{can}}}\cdot(N_{F/\mathbb{Q}}\mathfrak{p})^{sk}\cdot C'_{1},$$
    $$\widehat{\nu}(L_{p}(K_{\infty},f\otimes\chi_{F'/F}))=\frac{L(\frac{1}{2},\pi_{K}
    \otimes\chi_{F'/F}\otimes \nu)}{\Omega_{f\otimes\chi_{F'/F}}^{\mathrm{can}}}\cdot(N_{F/\mathbb{Q}}\mathfrak{p})^{sk}\cdot C'_{2}$$
    and
    $$\widehat{\nu}'(L_{p}(K'K_{\infty},f'))=\frac{L(\frac{1}{2},\pi'_{K'}\otimes \nu')}{\Omega_{f'}^{\mathrm{can}}}\cdot(N_{F'/\mathbb{Q}}\mathfrak{p}')^{sk}\cdot C'_{3}$$
    where $C'_{1}$, $C'_{2}$, $C'_{3}$ are all $p$-adic units. Using the same method as in Case 1, we
    derive that $\Omega_{f'}^{\text{can}}$ and $ \Omega_{f}^{\text{can}}\cdot\Omega_{f\otimes\chi_{F'/F}}^{\text{can}}$ differ a $p$-adic unit.

According to \cite{D2}, the congruence number for $\lambda_f$ and
Petersson inner product $  \langle f, f \rangle_{\mathrm{Pet}}$
remain the same if $f$ is twisted by a character, so we have
$$\Omega_{f}^{\text{can}}=\Omega_{f\otimes\chi_{F'/F}}^{\text{can}}.$$
Thus the canonical period of $f'$ is a $p$-adic unit times the
square of the canonical period of $f$. \qed

\begin{cor}\label{thm:mu}
 We have $$ \mu(L_{p}(F'K_{\infty},f'))=0. $$
\end{cor}

This is a direct consequence of Theorem \ref{cor:lv} and
\cite[Theorem 6.9]{Hung}. Note that this result cannot be derived
directly from $\mu(L_{p}(K'_{\infty},f'))=0$.

We end this section with  the following conjecture closely related
to Corollary \ref{thm:mu}.

\begin{conj}\label{conj-mu} Let $F$ be a totally real number field, $K$ a quadratic imaginary
extension of $F$. Let $F'$ be a finite totally real extension of
$F$, and $f_{F'}$ be a Hilbert modular form over $F'$ that is
ordinary at each place of $F'$ above $p$. Then
$$\mu(L_p(F'K_\infty, f_{F'}))=0.$$
\end{conj}
Corollary \ref{thm:mu} says that Conjecture \ref{conj-mu} holds when
$F'$ is a quadratic real extension of $F$ and $f_{F'}$ is the base
change of a Hilbert modular form over $F$ (that satisfies some
conditions). It seems impossible to use the method in \cite{CH18,
Hung} to prove Conjecture \ref{conj-mu}, though it works when
$F'=F$.

\end{document}